\documentclass{article}
\usepackage{amsmath}
\usepackage{amsfonts}
\usepackage{amssymb}
\usepackage{latexsym}
\usepackage{amsthm}
\usepackage{graphicx}
\usepackage{color}
\newtheorem{theorem}{Theorem}[section]
\newtheorem{lemma}[theorem]{Lemma}

\theoremstyle{definition}
\newtheorem{defn}[theorem]{Definition}
\newtheorem{example}[theorem]{Example}
\theoremstyle{remark}
\newtheorem{conj}[theorem]{Conjecture}
\newtheorem{ques}[theorem]{Question}

\numberwithin{equation}{section}
\newcommand{\E}{\mathbb{E}}

\newcommand{\M}{M^{\mathbb{N}}}
\makeatletter
\def\imod#1{\allowbreak\mkern10mu({\operator@font mod}\,\,#1)}
\makeatother

\begin{document}
\pagestyle{myheadings}
\markboth{I. Assani and K. Presser}{Survey of the Return Times Theorem}

\title{A Survey of the Return Times Theorem}
\author{Idris Assani (University of North Carolina at Chapel Hill) \\
 Kimberly Presser (Shippensburg University)}

\maketitle

\begin{abstract}
The goal of this paper is to survey the history, development and current status of the Return Times Theorem and its many extensions and variations.  Let $(X, \mathcal{F}, \mu)$ be a finite measure space and let $T:X \rightarrow X$ be a measure preserving transformation.  Perhaps the oldest result in ergodic theory is that of \textbf{Poincar\'e's Recurrence Principle} \cite{Poin} which states:

\begin{theorem}\label{Poin}
For any set $A \in \mathcal{F}$, the set of points $x$ of $A$ such that $T^nx \notin A$ for all $n > 0$ has zero measure.  This says that almost every point of $A$ returns to $A$.  In fact, almost every point of $A$ returns to $A$ infinitely often.
\end{theorem}

The \textbf{return time} for a given element $x \in A$,
$$r_A(x) = \inf\{k \geq 1: T^kx \in A\},$$
is the first time that the element $x$ returns to the set $A$.  This is visualized in the figure below.

\begin{center}
\includegraphics[height=2in]{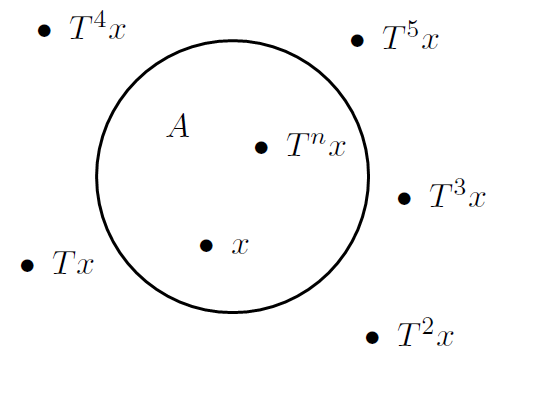}
\end{center}

By \textbf{Theorem \ref{Poin}}, there is set of full measure in $A$ such that all elements of this set have a finite return time.  Our study of the return times theorem asks how we can further generalize this notion.

\end{abstract}

\maketitle

\tableofcontents

\section{Origins}

As stated above in \textbf{Theorem \ref{Poin}}, $\mu$-a.e. $x$ returns to $A$ infinitely often.  One question to ask is how frequently this occurs.  Consider the time average $$\frac{1}{N}\sum_{n=1}^N \chi_A(T^nx).$$
This is a measure of how frequently the element $x$ returns to $A$.  We can then define a sequence recursively which characterizes the times that $x$ returns to $A$.

\begin{defn}\label{return}
The \textbf{return times sequence for $x$ with respect to the set $A$} is defined as
$$k_1(x,A) = \inf\{n:T^nx \in A\}$$
and
$$k_i(x,A) = \inf\{n > k_{i-1}(x,A):T^nx \in A\}.$$
\end{defn}

Thus the time average $\displaystyle \frac{1}{N}\sum_{n=1}^N \chi_A(T^nx)$ is simply a measure of the density of the sequence $k_i(x,A)$ in the set $\{1,...,N\}$.

Recall the following definition of an ergodic measure-preserving system.   \begin{defn}\label{ergodic}
The finite measure-preserving system $(X,\mathcal{F}, T, \mu)$ is \textbf{ergodic} if $T^{-1}(B) = B$ for $B \in \mathcal{F}$ implies $\mu(B) = 0$ or $\mu(B) = \mu(X)$.  In this case $T$ is called an \textbf{ergodic transformation}.
\end{defn}

Using \textbf{Birkhoff's Pointwise Ergodic Theorem} \cite{Birkhoff}, we know something more about the density of this sequence.

\begin{theorem}\label{Birk}
If $T$ is a measure preserving map on the finite measure space $(X,\mathcal{F},\mu)$ then for each real-valued $f\in L^1(\mu)$,
$$\lim_N\frac{1}{N}\sum_{n=1}^Nf\left(T^nx\right)= \E(f|Inv)(x),\mu-a.e.,$$
where $Inv = \{B \in \mathcal{F}:T^{-1}(B) = B\}$.  In particular, the limiting function has the same integral as $f$ and $\E(f|Inv)(x)=\int_Xfd\mu$ for all $x\in X$ when $T$ is ergodic.
\end{theorem}

For a given factor $\mathcal{Y}$ of $\mathcal{F}$, the notation $\E(f | \mathcal{Y})$ or $P_{\mathcal{Y}}$ refers to the conditional expectation of $f$ with respect to the factor $\mathcal{Y}$.  It should be noted that Birkhoff's original proof is when $f$ is a characteristic function (as in the case we are considering above) in the setting of a closed analytic manifold having a finite invariant measure.  Khintchine \cite{Khin} showed that \textbf{Theorem \ref{Birk}} remains true for an integrable $f$ on an abstract finite measure space.  The essence of the proof of Khintchine is that found in Birkhoff's proof.

Applying \textbf{Theorem \ref{Birk}} to our time average above, when $T$ is ergodic then for $\mu$-a.e. $x \in X$
$$\lim_N\frac{1}{N}\sum_{n=1}^N \chi_A(T^nx) = \int_X \chi_Ad\mu = \mu(A).$$
Thus, whenever $A$ has positive measure the sequence $k_i(x,A)$ has positive density for almost every $x \in X$.

\subsection{Averages along Subsequences}

The study of time averages can be extended by looking at the convergence of the averages along some sequence $k_i$:
\begin{equation}
\lim_{n}\frac{1}{n}\sum_{i=1}^{n}g(S^{k_i}y).\label{Average}
\end{equation}

In this case we want to know what properties of the $k_i$ would guarantee a.e. convergence.  In studying such sequences we can consider either deterministic sequences such as the sequence of squares or primes as well as randomly generated sequences such as those that come from looking at the values $n$ for which $\chi_A(T^nx) = 1$.  We are interested in finding when a randomly generated sequence will give convergence of the averages in equation (\ref{Average}).

In 1969, A. Brunel and M. Keane \cite{BK} use sequences of this type to determine a more powerful result concerning time averages along randomly generated subsequences.  The result of interest is Theorem 1 on page 3 of their paper.  They follow up this result by looking at norm convergence of similar averages as well.  It has been indicated by Michael Keane that the main idea for this return times average was a part of Antoine Brunel's thesis \cite{Brunel}.  Their result states

\begin{theorem}\label{Brunel}
If  $(Y,\mathcal{G},\nu,S)$ is a measure preserving system and if  $g \in L^1(\nu)$ and if $k_1, k_2, \ldots$ is a uniform sequence, then
$$\tilde{g}(y) = \lim_{n}\frac{1}{n}\sum_{i=1}^{n}g(S^{k_i}y)$$
exists almost everywhere and $\tilde{g}\in L^1(\nu)$.
\end{theorem}

A uniform sequence is a sequence $k_1, k_2, \ldots$ of natural numbers such that there exists

\begin{enumerate}
\item a strictly $L$-stable system $(X,\mathcal{F},\mu,T)$
\item a set $A \in \mathcal{F}$ such that $\mu(A)>0=\mu(\delta A)$ where $\delta A$ is the boundary of $A$
\item a point $x \in X$ such that $k_i = k_i(x,A)$ for each $i \geq 1$ (where $k_i(x,A)$ is as defined above).
\end{enumerate}

\begin{defn}
A measure preserving system $(X, \mathcal{F}, \mu, T)$ is a \textbf{strictly} $\mathbf{L}$\textbf{-stable system} if $X$ is a compact metric space and $T$ is a homeomorphism of $X$ such that
\begin{enumerate}
\item $T^n$ is an equicontinuous set of mappings, i.e. for any $\epsilon >0$ there exists a $\delta >0$ such that for $x,y \in X$, $d(x,y) < \delta$ implies $d(T^nx,T^ny)<\epsilon$ for any integer $n$ and
\item $X$ possesses a dense orbit under $T$, i.e., there exists some $x$ in $X$ such that $\{T^nx : n \in \mathbb{Z}\}$ is dense in $X$.
\end{enumerate}
\end{defn}

\begin{example}\label{Lstable}
Consider $X = D$ the unit disk $\{z \in \mathbb{C}: |z| = 1\}$ under the standard Borel measure.  This is a compact measure space and let $T:D \rightarrow D$ be the irrational rotation defined by $Tz = e^{2\pi i \alpha}z$ for some irrational $\alpha$.  Then $T$ is a homeomorphism of $D$.  For any $\epsilon > 0$ and $z_1, A_2 \in D$, if $|z_1 - z_2| < \epsilon$ then
$$|T^nz_1 - T^nz_2| = |e^{2\pi in\alpha} z_1 - e^{2\pi i n\alpha}z_2| = |z_1 - z_2| < \epsilon.$$
Therefore the $T^n$ form an equicontinuous set of mappings.  As $\alpha$ is irrational, $D$ possesses a dense orbit under $T$ for any $z \in D$.  Therefore this gives us a strictly $L$-stable system on which we can build a uniform sequence.  In fact, according to Brunel and Keane \cite{BK} every strictly L-stable system is homeomorphic to some ergodic rotation on a compact abelian group.
\end{example}

Therefore if our measure preserving system is a strictly $L$-stable system, then the sequence of return times $k_i(x,A)$ which was discussed in \textbf{Definition \ref{return}} is a uniform sequence.  Rewriting \textbf{Theorem \ref{Brunel}} in terms of the characteristic function of our set $A$ gives:
\begin{theorem}\label{Brunel2}
For any ergodic dynamical system $(X, \mathcal{F}, \mu, T)$ with $\mu(X)< \infty$ and $A \subset X$ with $\mu(A)>0$.  Then there exists a set $X_f$ in $X$ of full measure such that for any other measure preserving system $(Y,\mathcal{G},\nu,S)$ and any $g \in L^1(\nu)$ then
$$\lim_n\frac{1}{n}\sum_{i=1}^n\chi_A(T^nx)g(S^ny)$$
exists for $\nu$-a.e. $y$.
\end{theorem}

The interesting fact to note here and in the more general statement by Brunel and Keane is that the convergence is in a sense ``universal'' because it does not depend on the choice of the dynamical system $(Y,\mathcal{G},\nu,S)$ nor the function $g \in L^1(\nu)$.

Krengel's book \underline{Ergodic Theorems} \cite{KrengelET} highlights the generalization of the work of Brunel and Keane to the concept of universally good sequences.

\begin{defn}
A strictly increasing sequence $k_i \in \mathbb{N}$ (i.e. a subsequence of $\mathbb{N}$) is \textbf{universally good} with respect to a stated type of convergence (this could be norm convergence, pointwise convergence, etc.) if for any measure preserving system $(Y,\mathcal{G},\nu,S)$  and $g \in L^r(\nu)$, then
$$\frac{1}{n}\sum_{i=1}^{n}g(S^{k_i}y)$$
converges.
\end{defn}

For example, the result of Brunel and Keane can be restated to say that uniform sequences are universally good with respect to a.e. convergence for functions in $L^1(\nu)$.  The Blum and Hanson Theorem \cite{BH} shows that when $S$ is a mixing endomorphism and $(Y,\mathcal{G},\nu)$ is a probability space, then every strictly increasing sequence $k_i \in \mathbb{N}$ is universally good with respect to the $L^2(\nu)$-norm for functions $g$ in $L^2(\nu)$.

On the opposite side, the question of when sequences fail to be universal has also been an object of intense study.  Friedman and Ornstein \cite{FO} use a tower construction to create a strongly mixing operator $T$ for which the conditions in Blum and Hanson are not sufficient for guaranteeing universal a.e. convergence for functions in $L^1(\mu)$.  Krengel \cite{KrengelOTI} was able to construct ``universally bad sequences'' for which regardless of the measure preserving system chosen one can find a function $g$ in $L^1(\mu)$ for which the averages do not converge a.e..  Bellow \cite{BellowPOA} created a sequence which was universally good for a.e. convergence for functions in $L^p$ when $1 < p < \infty$, but at the same time was universally bad for a.e. convergence in $L^q$ when $1 \leq q < p$.  This result hints at the concerns about breaking the duality which we will discuss later in \textbf{Section \ref{Duality}}.

The above ``bad'' examples have all been sequences of zero density.  Along those lines, Friedman \cite{Fried} posed the question of whether or not it was true that sequences of positive density were universally good for a.e. convergence with respect to functions in $L^1(\mu)$.  Conze \cite{Conze} was able to prove that when $S$ is an automorphism with Lebesgue spectrum then sequences of positive lower density were universally good for a.e. convergence with respect to functions in $L^1(\mu)$, but he was also able to construct a counterexample to Friedman's question for $T$ without Lebesgue spectrum.

\subsection{Weighted Averages}

Returning to the concept of universally good sequences, we want to consider these sequences as ``weights'' on our average as follows.  When $n=k_i$ for some $k_i$ in our sequence we add up $g(S^ny)$ (weighting it by $1$) otherwise we weight it by $0$ and do not add the $g(S^ny)$ term.  This concept was extended to the convergence of more general weights in Bellow and Losert \cite{BL}.  They begin by defining a \textbf{good weight}.

\begin{defn}
A sequence of complex numbers $\mathbf{a} = (a_n)$ is a \textbf{good weight} in $L^p(\nu)$ for an operator $S$ on $L^p(\nu)$ if, for every $g \in L^p(\nu)$
$$\lim_N \frac{1}{N}\sum_{j=0}^{N-1}a_jg(S^jy)$$
exists $\nu$-a.e.
\end{defn}

Note that if our operator on $L^p(\nu)$ is simply the one induced by our measure preserving system $g \rightarrow g\circ S$, then this average is the same type we have been discussing before and with the weights  $a_n = 1$ if $n = k_i$ and $0$ otherwise.  Our ultimate goal is to find a weight for which convergence holds independently of the second dynamical system. Bellow and Losert define a \textbf{good universal weight} as a sequence which is a good weight for every operator induced by a measure preserving transformation on $(Y, \mathcal{G},\nu)$.  That this is equivalent to being a good weight for every Dunford -Schwartz operator or every operator induced by an ergodic transformation is shown in their Theorem 1.4 which is due to Baxter and Olsen \cite{BO}.  Thus the earlier result of Brunel and Keane says that for a uniform sequence $k_j$, then the sequence $a_i = 1$ if $i = k_j$ and $0$ otherwise is a good universal weight for a.e. convergence.

The next area of study then is to try and determine what types of sequences make universal good weights for a.e. convergence.  Section 3 of \cite{BL} shows that this result can be extended to show that a set of sequences $\mathcal{D}$ strictly containing the bounded Besicovitch sequences (which include the uniform sequences of Brunel and Keane) are good universal weights for a.e. convergence.  This is an extension of the Ryll-Nadzewski \cite{Ryll} work which shows that the bounded Besicovitch sequences are good universal weights for a.e. convergence.  Then Bellow and Losert prove that for a strictly $L$-stable system, $(X,\mathcal{F}, \mu, T)$ and any integrable $f$, the sequence $\mathbf{a} =\{f(T^nx)\}$ is bounded Besicovitch for all $x \in X$.

Bellow and Losert extend the results of Blum and Reich \cite{BR} and others to determine other characteristics of dynamical systems such as ergodicity, saturation or the spectral measure associated with the sequences which can lead to the sequence $\mathbf{a} = \{f(T^nx)\}$ being classified as a good universal weights for a.e. convergence.

In Theorem 5.4 of their paper, Bellow and Losert prove that if $T$ has Lebesgue spectrum, then for every $f \in L^{\infty}(\mu)$ there exists a set of full measure $X_f$ in $X$ such that for every $x \in X_f$, the sequence $\mathbf{a}=\{f(T^nx)\}$ is a good universal weight for a.e. convergence.  Bellow and Losert note that the result for K-automorphisms was previously shown using different methods by H. Furstenberg, M. Keane, J.P. Thouvenot and B. Weiss.

As an extension of the work of Bellow and Losert, one can consider a dynamical system $(X, \mathcal{F}, \mu, T)$ and function $f$ and ask what properties of the dynamical system are required for the sequence $\mathbf{a} = (f(T^nx))$ to be $\mu$-a.e. a good universal weight for a.e. convergence.  The answer to this question is what is standardly referred to as \textbf{Bourgain's Return Times Theorem}.

\begin{theorem}\label{RTT}
Let $(X,\mathcal{F},\mu,T)$ be an ergodic dynamical system of finite measure and $f \in L^{\infty}(\mu)$.  Then there exists a set $X_f \subset X$ of full measure such that for any other ergodic dynamical system $(Y,\mathcal{G}, \nu, S)$ with $\nu(Y)<\infty$ and any $g \in L^{\infty}(\nu)$:
$$\frac{1}{N}\sum_{n=1}^Nf(T^nx)g(S^ny)$$
converges $\nu$-a.e. for all $x \in X_f$.
\end{theorem}

Let $\left(X,\mathcal{F},\mu,T\right)$ and $\left(Y,\mathcal{G},\nu,S\right)$ be ergodic dynamical systems with $\mu(X)<\infty$, $\nu(Y)<\infty$, $f\in L^\infty(\mu)$ and $g\in L^\infty(\nu)$.  Applying Birkhoff's theorem to the multiple term Ces\`{a}ro average
$$\frac{1}{N}\sum_{n=1}^Nf\left(T^nx\right)g\left(S^ny\right)$$
it is known that this average converges $\mu \times \nu$-a.e..  However, our focus is on showing that this convergence is ``universal'' and thus independent of the choice of dynamical system $(Y,\mathcal{G}, \nu, S)$ and function $g$ which is a much stronger result.

\subsection{Wiener-Wintner Results}

While the above developments were taking place in the realm of measure theory, some related results were being discovered in the field of harmonic analysis.  In 1941, the publication of the \textbf{Wiener-Wintner Theorem} \cite{WW} gave a positive result for the Return Times Theorem when the second dynamical system is a rotation by $\alpha$.

\begin{theorem}\label{WW}
If  $(X,\mathcal{F},\mu,T)$ is a measure preserving system and if  $f \in L^1(\mu)$, then there exists a set $\tilde{X}$ of full measure in $X$ such that for $x\in \tilde{X}$ and for all $\theta \in \mathbb{R}$, the sequence
$$\frac{1}{N}\sum_{n=1}^{N}e^{2\pi in\theta}f\left(T^nx\right)$$
converges.
\end{theorem}

The proof by Wiener and Wintner relied on connecting the spectrum in the sense of Wiener with the point spectrum.  Unfortunately, there was an error in the proof of Wiener and Wintner.  However the theorem is true and has been proven multiple times using a wide variety of techniques.  The first correct version of the proof was given by H. Furstenberg \cite{FurstSPA} using the notions of joinings and generic points.  We will see later the role of joinings and generic points in the proof of the Return Times Theorem.  Another proof of the Wiener-Wintner Theorem which uses the Van der Corput's inequality \cite{KN} was given by Bourgain \cite{BourgainDRA} in 1990.  Actually Bourgain proved a stronger result referred to as \textbf{Bourgain's Uniform Wiener-Wintner Ergodic Theorem} which constitutes the proof that (1) implies (2) of the theorem below.

\begin{theorem}\label{UWW}
Given an ergodic dynamical system $(X, \mathcal{F}, \mu, T)$ and $f \in \mathcal{K}^{\perp}$ where $\mathcal{K}$ is the Kronecker factor (the $\sigma$-invariant algebra spanned by the eigenfunctions of $T$).  The following statements are equivalent
\begin{enumerate}
\item $f\in\mathcal{K}^{\perp}$.
\item For $\mu$-a.e. $x$,
$$\lim_N \sup_{\theta} \left| \frac{1}{N} \sum_{n=1}^N f(T^nx)e^{2\pi i n \theta} \right| = 0.$$
\end{enumerate}
\end{theorem}

\begin{proof}
Proof of (2) implies (1).  We can decompose $f$ into the sum $f_1+ f_2$ where $f_1\in \mathcal{K}$ and $f_2\in \mathcal{K}^ {\perp}.$  Using an orthonormal basis of eigenfunctions $e_j$ each with eigenvalue $e^{2\pi i\beta_j}$ we can write $f_1$ as $\sum_{j=1}^{\infty} \alpha_j(f) e_j$ where the convergence of the series is to be understood as being in $L^2$ norm.

Assume that $f_1\neq 0$, thus $\alpha_{j_0}(f) \neq 0$ for at least one $j_0$.  Then we would have by the Spectral Theorem
$$\int\bigg|\frac{1}{N}\sum_{n=1}^N f_1(T^nx)e^{-2\pi in\theta_{j_0}}\bigg|^2d\mu \leq \sup_{\theta}\bigg|\frac{1}{N}\sum_{n=1}^N f_1(T^nx)e^{2\pi in\theta}\bigg|^2.$$
The left hand side can be written as
$$\int\bigg|\sum_j^{\infty} \alpha_j(f)e_j\bigg(\frac{1}{N}\sum_{n=1}^N e^{2\pi in\theta_j}e^{-2\pi in\theta_{j_0}}\bigg)\bigg|^2d\mu = \sum_j^{\infty} |\alpha_j(f)|^2\bigg|\frac{1}{N}\sum_{n=1}^N
e^{2\pi in(\theta_j - \theta_{j_0})}\bigg|^2.$$
By taking the limit with $N$ we obtain a lower bound $|\alpha_{j_0}(f)|^2$ which is strictly positive if $f_1\neq 0$.
\end{proof}

The Wiener-Wintner Theorem while discovered independently from these return times averages follows as a consequence of the Return Times Theorem.  Thus any developments with regards to the Return Times Theorem have consequences for the Wiener-Winter result and thus may have other implications in the field of harmonic analysis.  Similarly, developments with the Wiener-Wintner Theorem, may give ideas for expanding the Return Times Theorem.  Our goal in this paper is to discuss the developments thus far and highlight some of the implications this may leave for future studies.

\section{Development}

In 1988, J. Bourgain \cite{RTT1} released a preprint of his proof of the Return Times Theorem.  His statement of the Return Times Theorem is as follows:

\begin{theorem}\label{RTTBourgain}
Let $(X,\mathcal{F},\mu,T)$ be a dynamical system with $\mu$ a finite positive measure and $T$ and ergodic measure-preserving transformation.  The generic return-time sequence is defined as $\Lambda_\omega = \{n \in \mathbb{Z}_+ : T^n\omega \in A\}$ for any set $A$ of positive measure and any point $\omega$ in $X$.  For almost all $\omega$ the sequence $\Lambda_\omega$ satisfies the pointwise ergodic theorem.    Given any dynamical system $(Y,\mathcal{G},\nu, S)$ where $\nu$ is a finite measure, the averages
$$\frac{1}{|\Lambda_N|}\sum_{n\in \Lambda_N} S^ng$$
converge $\nu$-almost surely for any $g$ in $L^1(\nu)$ where $\Lambda_N = \Lambda_{\omega} \cap [1,N]$.
\end{theorem}

This important result impacted both the fields of ergodic theory and harmonic analysis and fully generalized the work with good universal weights which was presented in the work of Bellow and Losert \cite{BL}.  This initial fifty-two page proof used difficult techniques from harmonic analysis which made it challenging to analyze the proof, nevertheless to extend those techniques to proving other convergence questions.  In 1989, J. Bourgain, H. Furstenberg, Y. Katznelson and D. Ornstein \cite{RTT2} published a more concise proof of the Return Times Theorem (found in the appendix to Bourgain's ``Pointwise ergodic theorems for arithmetic sets'') which utilized several key ergodic theory concepts such as the Rohlin Towers Lemma \cite{DGS}.  One assumption that they make in their argument is that the function $f$ has finite range.  A discussion of how to extend to a function $f$ which does not have finite range along with a detailed analysis of their argument can be found in either \cite{AssaniWWE} or \cite{Collins}.

\subsection{The BFKO proof of Bourgain's Return Times Theorem}

One of the key techniques used in the proof is to study the set of points on which the theorem holds true.  To do this we begin by looking at the set of points for which the Birkhoff averages converge.  This concept was introduced by Furstenberg \cite{FurstRIE} in the setting of regular measure-preserving systems.

\begin{defn}
A measure-preserving system $(X, \mathcal{F}, \mu, T)$ is call \textbf{regular} if the space $X$ is compact metrizable, the $\sigma$-algebra $\mathcal{F}$ is Borel and the transformation $T$ is continuous.
\end{defn}

\begin{defn}\label{FGen}
Let $(X, \mathcal{F}, \mu, T)$ be a regular measure-preserving system, $x_0 \in X$ and $\mu \in \mathcal{M}_T(X)$ (the $T$-invariant measures on $X$).  A point $x_0 \in X$ is \textbf{a generic point for $\mu$} if
$$\frac{1}{N} \sum_{n=1}^N f(T^nx_0) \rightarrow \int f d\mu$$
for every continuous function $f \in \mathcal{C}(X)$.
\end{defn}

It is known that every separable measure preserving system is equivalent to a regular one (see \cite{FurstRIE}), thus for our purposes we use an extended definition of generic to the context of a general measure preserving system.  In this case the genericity is dependent upon the choice of the function $f$.

\begin{defn}\label{BirkGen}
Let $(X, \mathcal{F}, \mu, T)$ be an ergodic measure preserving system.  Let $f$ be an integrable function defined on $X$.  A point $x_0 \in X$ is \textbf{generic for $f$} if
$$\frac{1}{N} \sum_{n=1}^N f(T^nx_0) \rightarrow \int f d\mu$$
\end{defn}

Thus Birkhoff's Pointwise Ergodic Theorem implies that for any $f \in L^1(\mu)$, $\mu$-a.e. $x \in X$ is \textbf{generic for $f$}.

The proof in \cite{RTT2} begins by decomposing our function $f$ with respect to the Kronecker factor $\mathcal{K}$ in order to handle the two factors separately in the proof.  When $f \in \mathcal{K}$ the statement follows fairly easily from \textbf{Theorem \ref{WW}}.  When $f \in \mathcal{K}^{\perp}$ (and has finite range) the authors establish results which describe the set of full measure $X_f$ on which the averages converge.  This set is namely the intersection of the set of full measure on which \textbf{Theorem \ref{Birk}} holds for $f$ (because of the possibility that one chooses $g \equiv 1$) and the sets
$$X_1=\left\{\lim_N \frac{1}{N} \sum_{n=1}^N f(T^nx)\overline{f(T^ny)} = 0 \textrm{ for } \mu\textrm{-a.e. }y\right\}$$
and
$$X_2=\bigcup_{n=1}^{\infty}\left\{x \in X: x \textrm{ is generic for } \chi_A \circ \Gamma_1 \textrm{ for any } A \subset F^n\right\}$$
where $F$ is the range of $f$ and $\Gamma_1:X \rightarrow F^n$ is defined by
$$\Gamma_1(x) = (f(T^2x),...,f(T^{n+1}x)).$$
As $f \in \mathcal{K}^{\perp}$ it has continuous spectral measure and thus the set $X_1$ has full measure.  The other sets are of full measure by \textbf{Theorem \ref{Birk}}.

The authors then make the assumption that \textbf{Theorem \ref{RTT}} does not hold on this set $X_f$.  Thus there exists some other ergodic dynamical system $(Y,\mathcal{G}, \nu, S)$ and $g \in L^{\infty}(\nu)$ such that the set
$$B =\left\{y \in Y:\limsup_{N}\left|\frac{1}{N}\sum_{n=1}^N f(T^nx)g(S^ny)\right|>0\right\}$$
has positive measure.  Without loss of generality this can be reduced to the situation where for some positive $a$ a set

\begin{equation}\label{Rohlin}
B_1=\left\{y \in Y: \lim_{k}\sup_{N \geq k} Re \left(\frac{1}{N}\sum_{n=1}^Nf(T^nx)g(S^ny)\right)>\frac{a}{2}\right\}
\end{equation}

which has positive measure.  Using the Rohlin Tower Lemma, given any $\delta > 0$ one can find an integer $K$ and a set $B_2 \subset B_1$ of positive measure such that the $S(B_2), S^2(B_2), \ldots, S^K(B_2)$ are pairwise disjoint and cover $B_2$ up to a set of measure less than $\frac{\delta}{3}$.  Using the set $B_1$ and the properties in the sets $X_1$ and $X_2$ one can create a sequence of properly spaced ranges $R_j=(L_j, M_j)$ on which the points behave poorly with respect to the return times average, but well with respect to the averages used in defining $X_1$ and $X_2$.  Using these ranges and the set $\tilde{B}$ one can create sequences $(c_n(y))_{n=1}^{N_o}$ (for some large enough $N_o$) which are the sum of $J$ layers $(c_n^j(y))_{j=1}^J$ which have the following properties:

\begin{enumerate}
\item For all $j$, $n$ and $y$, the $c_n^j(y)$ are uniformly bounded.
\item For $j_1 \neq j_2$, $\left|\frac{1}{N}\sum_{n=1}^Nc_n^{j_1}(y)\overline{c_n^{j_2}(y)}\right| < \delta$ and
\item $Re\left(\frac{1}{N_o}\sum_{n=1}^{N_o}c_n^j(y)g(S^ny)\right)>\frac{a}{2}-\delta$, for $j=1, \ldots, J$.
\end{enumerate}

This leads to a contradiction, thus the assumption that \textbf{Theorem \ref{RTT}} does not hold on the set $X_f$ is false.  Using Ergodic Decomposition one can show that the theorem holds true if either $T$ or $S$ is a measure preserving transformation which is not ergodic.

\subsection{Extensions of the Return Times Theorem}

\begin{defn}
We will say that the \textbf{return times theorem holds for the pair $(L^s, L^t)$} if for all $f \in L^s(\mu)$ we can find a set of full measure $X_f$ such that for each $x \in X_f$ for all measure preserving systems $(Y, \mathcal{G}, \nu, S)$ and for all $g \in L^t(\nu)$ the averages
$$\frac{1}{N}\sum_{n=1}^Nf(T^nx)g(S^ny)$$
converge $\nu$-a.e.
\end{defn}

As $L^{\infty}$ is dense in $L^1$,  using the Banach Principle you can show that the return times theorem holds for the pair $(L^1, L^{\infty})$ or $(L^{\infty}, L^1)$.  Using H\"{o}lder's Inequality one can show that the return times theorem holds for $(L^p, L^q)$ where $\frac{1}{p}+\frac{1}{q} = 1$.  Using the Banach Principle we can show that this is true when $\frac{1}{p}+\frac{1}{q} < 1$.  The case where $\frac{1}{p}+\frac{1}{q} > 1$ is explored below in \textbf{Section \ref{Duality}}.

Ornstein and Weiss \cite{OW} studied the set of points on which the Return Times Theorem holds in greater detail.  They begin by rephrasing the Return Times Theorem in the following way.

\begin{theorem}
If $(X,\mathcal{F},T, \mu)$ is an ergodic dynamical system and $B \in \mathcal{F}$ has positive measure, then for $\mu$-a.e. $x_0 \in X$, the sequence $\left\{n \in \mathbb{N}:T^nx_0 \in B\right\}$ is a good universal sequence for the return times theorem, i.e. for any finite measure preserving system $(Y, \mathcal{G}, \nu, S)$ and $g \in L^1$ we have
$$\lim_K \frac{1}{K}\sum_{k=1}^K g(S^{n_k}y) = \int g d\nu$$
for $\nu$-a.e. $y \in Y$.
\end{theorem}

The first extension that they present is a previously unpublished result of D. S. Ornstein, B. Weiss, H. Furstenberg, M. Keane and J.-P. Thouvenot.

\begin{theorem}\label{OWFKT}
If $(X,\mathcal{F}, \mu, T)$ has completely positive entropy, $B \in \mathcal{F}$ with positive measure and $x_0 \in X$ is generic for $\chi_B$, then the sequence $n_1 < n_2 < n_3 < \ldots$ of times of successive visits of $x_0$ to $B$ is a good sequence for the return times theorem.  Thus for any finite measure preserving system $(Y, \mathcal{G}, \nu, S)$ and $g \in L^1(\nu)$ we have
$$\lim_K \frac{1}{K}\sum_{k=1}^K f(T^{n_k}y) = \int g d\nu.$$
\end{theorem}

As noted above in the creation of the set $X_f$ in the BFKO proof of \textbf{Theorem \ref{RTT}}, more was needed than the genericity of $x$ to satisfy the Return Times Theorem.  Thus one cannot extend \textbf{Theorem \ref{OWFKT}} to a transformation which is merely ergodic using points which are merely generic with respect to the function $\chi_B$.  Ornstein and Weiss \cite{OW} describe the extra conditions on the point $x_0$ which are necessary in order for \textbf{Theorem \ref{OWFKT}} to hold when $T$ is an ergodic transformation.

\begin{defn}
Let $(X, \mathcal{F}, \mu, T)$ be an ergodic measure-preserving transformation and $B$ a set of positive measure in $\mathcal{F}$.  A point $x_0 \in X$ is \textbf{self-sampling for $\chi_B$} if for $\mu$-a.e. $x \in X$
\begin{equation}
\lim_N \frac{1}{N} \sum_{i=1}^N \chi_B(T^ix_0)\chi_B(T^ix) = \mu(B)^2. \label{ss}
\end{equation}
\end{defn}

This condition is necessary for $x_0$ in order to create a good sequence for the Birkhoff ergodic theorem and the BFKO proof shows that it is a sufficient condition.  To demonstrate the difference between genericity and self-sampling, Ornstein and Weiss present the following example.

\begin{example}\label{OWx}
Let $(X, \mathcal{F}, T, \mu)$ be a transformation with $-1$ in the spectrum so that there is a set $B$ of measure $\frac{1}{2}$ with $\mu(TB \cap B)=0$.  A point $x_0$ that visits $B$ at the following times $i$
\begin{itemize}
\item if $(2n)! \leq i < (2n+1)!$ and $i$ is even,
\item if $(2n+1)! \leq i < (2n+2)!$ and $i$ is odd,
\end{itemize}
will be generic for $\chi_B$, but for $\mu$-a.e. $x \in X$ condition equation (\ref{ss}) will fail to hold.
\end{example}

This example has some discrete spectrum.  We may wish to consider dynamical systems which are restricted from having discrete spectrum.  Consider the following definition.

\begin{defn}
A measure preserving system $(X, \mathcal{F}, T, \mu)$ is \textbf{weakly mixing} if $1$ is the only eigenvalue of $T$.  That is to say that the Kronecker factor is reduced to the set of constant functions.
\end{defn}

Ornstein and Weiss posed the question of whether or not genericity is necessary and sufficient for \textbf{Theorem \ref{OWFKT}} in the case of a weakly mixing dynamical system.  This question is still open.  The convergence of return times averages in case of weakly mixing dynamical systems will be discussed in greater detail in \textbf{Subsection \ref{Unique}}.

The rest of Ornstein and Weiss's paper \cite{OW} is devoted to extending the Return Times Theorem to a certain class of groups.  It consists of those groups $G$ which have a sequence of finite sets satisfying
\begin{itemize}
\item $A_1 \subset A_2 \subset \cdots$, and $\cup_{n=1}^{\infty} A_n = G$
\item for all $g \in G$, $\lim_n |gA_n \Delta A_n| /|A_n| = 0$
\item there is a constant $M$ such that for all $n$
$$|A_n^{-1}A_n| \leq M|A_n|.$$
\end{itemize}
That is to say $G$ is an amenable group with F\o lner sequence $A_n$ satisfying
$$\overline{\lim_n}|A_n^{-1}A_n|/|A_n| < \infty.$$
Tempelman \cite{Tempel} has shown that this class of groups satisfy Birkhoff's Pointwise Ergodic Theorem and thus they are a nice class to look at with respect to return times.

\subsection{Unique Ergodicity and the Return Times Theorem}\label{Unique}

Connecting this to our earlier discussion, Bourgain's Return Times Theorem \textbf{Theorem \ref{RTT}} proves that for any ergodic dynamical system $(X,\mathcal{F},\mu,T)$ of finite measure and $f \in L^{\infty}(\mu)$ for $\mu$-a.e. $x \in X$ the sequence $\mathbf{a} = f(T^nx)$ is a universally good weight for a.e. convergence with functions from $L^{\infty}$.  We would like to see if we can establish other properties for dynamical systems which will ensure that the sequence $\mathbf{a} = f(T^nx)$ forms a universal good weight for a.e. convergence.  We begin by looking at the concept of uniquely ergodic dynamical systems.

\begin{defn}
A transformation $T$ is \textbf{uniquely ergodic}, if there is only one measure $T$-invariant probability measure on $X$.  For example, the irrational rotation discussed in \textbf{Example \ref{Lstable}} is an example of a uniquely ergodic dynamical system.
\end{defn}

The importance of uniquely ergodic systems is summarized by the following result of Jewett \cite{Jewett} and Krieger \cite{Krieger}.

\begin{theorem}
Let $(X, \mathcal{F}, T, \mu)$ be an ergodic measure preserving system of the nonatomic Lebesgue probability space $(X, \mathcal{F}, \mu)$.  There exists a uniquely ergodic standard system $(Y,\mathcal{G}, S, \nu)$ which is isomorphic to $(X, \mathcal{F}, T, \mu)$.
\end{theorem}

In the case of uniquely ergodic operators there is a much stronger version of \textbf{Theorem \ref{Birk}} found in \cite{KB}.

\begin{theorem}\label{unique}
Let $(X, \mathcal{F}, T, \mu)$ be a uniquely ergodic measure preserving system.  Then for any $f \in \mathcal{C}(X)$
$$\frac{1}{N} \sum_{n=1}^N f(T^nx) \rightarrow \int f d\mu$$
uniformly in $X$.
\end{theorem}

Therefore if $T$ is a continuous map of the compact metric space $X$ to itself, then $T$ is uniquely ergodic (with unique ergodic measure $\mu$) if and only if every point of $X$ is generic for the measure $\mu$.

Two versions of \textbf{Theorem \ref{WW}} for uniquely ergodic transformations were proven independently by I. Assani \cite{AssaniUWW,AssaniWWE} and E. A. Robinson \cite{Robinson} and are stated below.  Extensions of the results of Assani and Robinson can be found in Walters \cite{WaltersTWW}, Santos and Walkden \cite{SW}, Lenz \cite{Lenz09, Lenz09-1} and M. Schreiber \cite{SC12}.

\begin{theorem}
\cite{AssaniUWW,AssaniWWE} Let $(X, \mathcal{F},T, \mu)$ be a standard uniquely ergodic system.  If $f \in \mathcal{C}(X) \cap \mathcal{K}^{\perp}$ then
$$\lim_N \sup_x \sup_t \left|\frac{1}{N} \sum_{n=1}^Nf(T^nx)e^{2\pi int}\right| = 0.$$
\end{theorem}

\begin{theorem}
\cite{Robinson} Let $(X, \mathcal{F}, T, \mu)$ be a standard uniquely ergodic system. Let $M_T$ be the set of eigenvalues for $T$ and let $C_T$ be the subset of eigenvalues with a corresponding continuous eigenfunction.  If $\lambda \in C_T \cup (M_T)^c$ then for every continuous function $f$ the averages
$$\frac{1}{N}\sum_{n=1}^N f(T^nx) \lambda^n$$
converge uniformly in $x$.
\end{theorem}

Using the ranges construction in the BFKO proof, E. Lesigne, C. Mauduit and B. Moss\'e \cite{LMM}, describe a criteria for a sequence to form a universal good weight for a.e. convergence to $0$ for functions in $L^1$.  Note the connection between these criteria and the self-sampling definition from \cite{OW}.

\begin{theorem}\label{LMM}
\textbf{Part 1:} Given a bounded sequence $u_n$ of complex numbers such that for all $\delta > 0$ there exists some $L_{\delta} > 0$ such that for all $L > L_{\delta}$ there exists some $M_{\delta, L} > 0$ such that for all  $M > M_{\delta, L}$
$$\frac{1}{M} \#\left\{m \in [0,M] : \forall n \in [L_{\delta}, L], \left| \frac{1}{n}\sum_{k=0}^{n-1} u_{m+k}\overline{u_k}\right|< \delta \right\} > 1 - \delta.$$
Then for any probability measure preserving system $(Y, \mathcal{G}, \nu, S)$ and every $f \in L^1(\nu)$,
\begin{equation}
\lim_n \frac{1}{n}\sum_{k=0}^{n-1} u_k \cdot g(S^ky) = 0 \label{crit1}
\end{equation}
for $\nu$-a.e. $y$.

\textbf{Part 2:}
Given a probability measure preserving system $(X, \mathcal{F}, T, \mu)$ and a bounded measurable function $f$ of $Y$ which is in $\mathcal{K}^{\perp}$ (the orthocomplement of the Kronecker factor), then for almost every $x \in X$ the point $x$ is generic for the function $f$ in the dynamical system and
\begin{equation}
\lim_n \frac{1}{n}\sum_{k=0}^{n-1}f(T^kx)\cdot \overline{f(T^k x^{\prime})} = 0 \label{crit2}
\end{equation}
for $\mu$-a.e. $x^{\prime}$.

Finally, if a point $x$ satisfies  equation (\ref{crit2}), then the sequence $u_n = f(T^nx)$ satisfies equation (\ref{crit1}).
\end{theorem}

Using this criterion, given that unique ergodicity as seen in \textbf{Theorem \ref{unique}} gave us the Birkhoff's Pointwise Convergence Theorem (with uniform convergence) for all $x \in X$, we pose the following questions initially posed by Assani and Host respectively.

\begin{ques}
If $(X,\mathcal{F}, \mu, T)$ is a uniquely ergodic probability system can we strengthen Bourgain's Return Times Theorem to give uniform convergence for all $x \in X$.  That is to say, is $\mathbf{a} = f(T^nx)$ a good universal weight for a.e. pointwise convergence for functions in $\mathcal{C}(X)$?
\end{ques}

\begin{ques}
If $(X,\mathcal{F}, \mu, T)$ is a weakly mixing dynamical system and $f \in \mathcal{C}(X)$.  Is the sequence $f(T^nx)$ a good universal weight for the pointwise convergence in $L^1$ for each $x \in X$?
\end{ques}

In Proposition 5.3 of \cite{AssaniWWE}, Assani gives a partial answer to these questions.  For this result we need the following definition.

\begin{defn}
A weakly mixing measure preserving system $(X, \mathcal{F}, T, \mu)$ is said to have \textbf{Lebesgue spectrum} if the spectral measure of each function $f \in L^2(\mu)$ with $\int f d\mu = 0$, $\sigma_f$, is absolutely continuous with respect to Lebesgue measure.
\end{defn}

\begin{theorem}\cite{AssaniWWE}
Let $(X, \mathcal{F}, T, \mu)$ be a uniquely ergodic system which is weakly mixing with Lebesgue spectrum.  Consider $f \in \mathcal{C}(X)$.  Then for all $x \in X$, the sequence $\mathbf{a} = f(T^nx)$ is a good universal weight for the a.e. pointwise convergence in $L^1$.
\end{theorem}

\subsection{A Joinings Proof of the Return Times Theorem}

Around the same time as the above work was being done, D. Rudolph was working on a proof of \textbf{Theorem \ref{RTT}} as well.  His proof which was published in \cite{RudolphAJP} transfers the problem to the study of shift invariant measures defined on the space of sequences.  He then uses joinings of these measures to reach the same conclusion.  Rudolph acknowledges that his proof follows essentially the same path as the BFKO proof \cite{RTT2}, using the characteristics of genericity and self-sampling to give a proof by contradiction.

Rudolph begins his proof of Bourgain's Return Times Theorem (\textbf{Theorem \ref{RTT}}) by determining the set of full measure in $X$ on which the theorem should hold true.  He calls the set $G(f)$.  The set $G(f)$ is created using the definition of a product-null function and aspects of genericity with respect to $f,f$ (that is the self-sampling property).  This is akin to the construction of the set $X_1$ and ultimately $X_f$ in the BFKO proof.  It is exactly the set of points which satisfies equation (\ref{crit2}) in the Theorem of Lesigne, Mauduit and Moss\'e (\textbf{Theorem \ref{LMM}}).

\begin{defn}
Suppose that $(X, \mathcal{F}, \mu, T)$ is a dynamical system  and $f \in L^{\infty}(\mu)$ with $\int  f d\mu = 0$.  A function $f$ is \textbf{product-null} if for $\mu \times \mu$-a.e. $(x_1, x_2)$,
$$\lim_n\frac{1}{n}\sum_{i=0}^{n-1}f(T^i(x_1))\overline{f}(T^i(x_2))=0.$$
\end{defn}

The link between product-null functions, spectral measure and the Kronecker factor is outlined in the following lemma of Rudolph.  Note that this lemma actually corrects the corresponding statement in the BFKO proof.

\begin{lemma}
For $f \in L^{\infty}(\mu)$ with $\int  f d\mu = 0$, the following are equivalent:
\begin{enumerate}
\item $f$ is product-null.
\item The spectral measure of $f$ is non-atomic.
\item $f$ is orthogonal to all eigenfunctions of $T$, i.e. $f \in \mathcal{K}^{\perp}$.
\item For $\mu$-a.e. ergodic component $\hat{\mu}$ of $\mu \times \mu$
$$\int f(x_1)\overline{f}(x_2)d\hat{\mu}=0.$$
\end{enumerate}
\end{lemma}

As in the BFKO proof, Rudolph reduces the problem using various techniques to a simpler setting.  Using spectral decomposition, he assumes ergodicity.  Define $H$ as the subset in $L^p(\mu)$ of functions $f$ for which the theorem holds true.  Using H\"{o}lder's inequality and properties of limits, one can show that $H$ is a closed subspace.  In a proof akin to the BFKO argument, Rudolph shows that all of the eigenfunctions lie in $H$.

What if $H$ were not all of $L^p(\mu)$?  Then there would be have to be a product null $f$ which was not in $H$.  This is deduced by taking a function which is not in $H$ and using the decomposition of its measure into atomic and continuous parts to create a product null function which must also not be in $H$.  This brings us back to the earlier classification of $G(f)$ which was based off of the averages for product null functions.  Thus the theorem has been reduced to looking only at product null functions and averages which converge to $0$ rather than converge to some general limit.

Let $H^*$ be all of those functions $g \in L^1(\nu)$ for which the return times averages do not converge to 0 $\nu$-a.e. for this $f$.  Using a similar argument as for $H$, one can show this is a closed subspace of $L^1(\nu)$.  Thus there is a product null $f$ and a function $g$ such that for some element $x \in G(f)$ there is a set of positive measure $B$ on which

\begin{equation}\label{Contra}
\overline{\lim}_n\left|\frac{1}{n}\sum_{i=0}^{n-1}f(T^ix)\overline{g}(S^iy)\right| \geq a
\end{equation}
for some $a>0$.  This is akin to the set $B_1$ and $a$ constructed in equation (\ref{Rohlin}) of the BFKO proof.

The key technique of Rudolph's argument is found in his Lemma 10 where he transfers over the above assumption to the language of measures and joinings.  This gives you a series of measures which create a contradiction with the product null characterization and genericity for these points.  Because the function $f$ is product-null, you have points that are $f$,$f$ generic for a measure where the integral of the points is $0$, but our contradictive assumption will allow us to show that because of equation (\ref{Contra}) this integral is bounded below by $a$.

In studying the current status of the study of return times averages, we will look at the problem from three different angles: the multiterm case, characteristic factors, and breaking the duality.  Each of these areas has made significant progress since the publishing of the proof of the Bourgain's Return Times Theorem (\textbf{Theorem \ref{RTT}}).  For each topic, we will present a bit of the historical background, the current results and some open questions still under consideration.

\section{The Multiterm Return Times Theorem}\label{Multi}

In 1989, D. Rudolph visited the Department of Mathematics at the University of North Carolina at Chapel Hill while he was working on his joining proof of J. Bourgain's Return Times Theorem.  Discussions of these topics continued as I. Assani visited Maryland in spring of 1990 and D. Rudolph returned to Chapel Hill for the entire spring of 1991.  In their discussions, I. Assani demonstrated how one could extend the Return Times Theorem result to pairs of functions satisfying the H\"olderian duality and mentioned as a follow-up the following problems:

\begin{enumerate}
\item The break of duality for the Return Times Theorem
\item The Multiterm Return Times Theorem
\end{enumerate}

At a conference held at the University of North Carolina at Chapel Hill in the fall of 1991, I. Assani again mentioned these two open questions.  Specifically, he raised the question, ``If the Kronecker factor characterizes the functions for which the return times limit is not zero, then what could be the factor which would characterize the three term return times or more generally the $h$ term return times theorem?''  An answer to this question is D. Rudolph's Multiterm Return Times Theorem that we will discuss in this section.\footnote{See Math Review MR1489899 (99c:28055) for more information on the historical development.} The second question on the break of duality is discussed in \textbf{Section \ref{Duality}}. Note that evidence of the validity of such result was first announced for the weakly mixing case in \cite{AssaniMRA}.

When we say multiterm return times we are looking for $\mu_i$-a.e. convergence in the same universal sense for averages of the form
$$\frac{1}{N} \sum_{n=1}^{N}\prod_{i=1}^H f_i(T_i^nx_i)$$
(universal sense meaning that the sets of full measure $X_{f_i}$ associated with each bounded function $f_i$ depend only on those $j$ with $1\leq j<i$).

As the Wiener-Wintner Theorem (\textbf{Theorem \ref{WW}}) was a useful tool in the BFKO proof of the Return Times Theorem (\textbf{Theorem \ref{RTT}}), it was a logical first step to look at a multiterm version of the Wiener-Wintner Theorem.

\begin{theorem}\label{ALR}
Let $(X,\mathcal{F}, \mu, T)$ be a measure preserving system and $f \in L^2(\mu)$.  For $\mu$-a.e. $x$, for any measure preserving system $(Y, \mathcal{G}, \nu, S)$ and any $g \in L^2(\mu)$ for $\nu$-a.e. $y$ and for all $\theta \in \mathbb{R}$, the sequence
$$\frac{1}{N}\sum_{n=0}^{N-1}f(T^nx)g(S^ny)e^{in\theta}$$
converges.
\end{theorem}

If $(X, \mathcal{F}, \mu, T)$ is weakly mixing, it was shown in \cite{AssaniUWW} that \textbf{Theorem \ref{ALR}} follows as a simple consequence of Bourgain's Return Times Theorem (\textbf{Theorem \ref{RTT}}).  The general proof of \textbf{Theorem \ref{ALR}} was obtained by E. Lesigne, D. Rudolph and I. Assani \cite{ALR}.  Their argument uses disintegration of measures.  As such they reduce the proof to the case of a regular measure-preserving system.  The $\sigma$-algebra used to characterize convergence (in the same role as the Kronecker factor played above) was exactly the one created by J.-P. Conze and E. Lesigne \cite{CL-TEP, CL-SUT} while proving the convergence in $L^1$ norm of the multiple recurrence averages for totally ergodic systems
\begin{equation}
\frac{1}{N}\sum_{n=1}^N f_1\circ T^nf_2 \circ T^{2n}f_3 \circ T^{3n}.\label{CLAverage}
\end{equation}
This result showed a link between the study of the multiterm return times and Furstenberg nonconventional ergodic averages. Note that, at the time, for a general measure preserving system the norm convergence of the averages in equation (\ref{CLAverage}) was not established.  The connection between characteristic factors and return-time phenomena will be discussed in more detail in \textbf{Section \ref{Character}}.

The convergence of the multiterm return times averages was obtained in 1993 and sent for publication in early 1994 for all positive integers $H$ and $L^1$ i.i.d. random variables by I. Assani in \cite{AssaniSLF}.  This was the first multiterm return time theorem obtained for all positive integers $H$. In 1998, D. Rudolph \cite{RudolphFGS} proved the Multiterm Return Times Theorem for bounded functions or functions which satisfy H\"older's inequality, $(L^{p_i})$ where $1\leq i\leq K$ and $\sum_{i=1}^K \frac{1}{p_i} \leq 1$.

\begin{theorem}\label{MultRTT}
Let $k \in \mathbb{Z}^+$.  For any dynamical system $(X_0,\mathcal{F}_0, T_0, \mu_0)$ and any $f_0 \in L^{\infty}(\mu)$, there exists a set of full measure $X_{f_0}$ in $X_0$ such that if $x_0 \in X_{f_0}$ for any other dynamical system $(X_1,\mathcal{F}_1, T_1, \mu_1)$ and any $f_1 \in L^{\infty}(\mu_1)$ there exists a set of full measure $X_{f_1}$ in $X_1$ such that if $x_1 \in X_{f_1}$ then $\ldots$ for any other dynamical system $(X_{k-1},\mathcal{F}_{k-1}, T_{k-1}, \mu_{k-1})$ and any $f_{k-1} \in L^{\infty}(\mu_{k-1})$ there exists a set of full measure $X_{f_{k-1}}$ in $X_{k-1}$ such that if $x_{k-1} \in X_{f_{k-1}}$ for any other dynamical system $(X_{k},\mathcal{F}_{k}, T_{k}, \mu_{k})$ the average
$$\frac{1}{N}\sum_{n=1}^N f_1(T_1^{n}x)f_{2}(T_{2}^{n}x_{2})f_{3}(T_{3}^{n}x_{3})\cdots f_{k}(T_{k}^{n}x_{k})$$
converges $\mu_{k}$-a.e..
\end{theorem}

Note that putting this in the context initially discussed, this result states that for any $j < k$, the random sequences $f_1(T_1^ix_1)f_2(T_2^ix_2) \ldots f_j(T_j^ix_j)$ are universal good weights for a.e. convergence.

One key feature of this argument is that the author proves his result without having to create higher-order $\sigma$-algebras. By converting the question at hand to the setting of measures on $\M$ where $M$ is a compact metrizable space, he avoids the method of factor decomposition exploited above in the BFKO proof of the Bourgain's Return Times Theorem.  The goal of Rudolph's paper is to construct an inductive argument to deduce the Multiterm Return Times Theorem from Bourgain's Return Times Theorem for two terms.  If it were true that whenever $a_i$ was a universal good weight then the sequence $a_if(T^ix)$ was a universal good weight, then the induction would follow directly from Bourgain's Return Times Theorem.  This is not the case as is demonstrated in the example presented in \cite{RudolphFGS} which is described below.

\begin{example}
Construct a bounded sequence $a_i$ with the following two properties:
\begin{enumerate}
\item For any measure preserving system $(Y, \mathcal{G}, \nu, S)$ and $g \in L^{\infty}(\nu)$
$$\frac{1}{N} \sum_{i=1}^{\infty}a_ig(S^iy) \rightarrow 0$$
but
\item there is a process $(X, \mathcal{F}, \mu, T)$ and $f$ so that for $\mu$-a.e. $x$ the sequence $a_if(T^ix)$ will look like the sequence in (\textbf{Example \ref{OWx}}) \cite{OW} which was shown to be not a good universal weight.
\end{enumerate}
\end{example}

Therefore, Rudolph attacks the problem by looking at another property of sequences which he will define as fully generic from which the inductive argument can be made.  His proof involves four major steps.  First, the setting is transformed to the language of measures on a compact metrizable space and the definition of fully generic will be made.  Secondly, it is shown that for any shift invariant measure, $\mu$, $\mu$-a.e. sequence is fully generic.  Next, he proves the induction step through his Theorem 1 which states:

\begin{theorem}\label{RT1}
Suppose $\vec{m}_0 \in M_0^{\mathbb{N}}$ is fully generic.  For a second compact metrizable space $M_1$ consider
$$A_{\vec{m}_0} = \left\{\vec{m}_1 \in M_1^{\mathbb{N}} : (\vec{m}_0, \vec{m}_1) \in (M_0 \times M_1)^{\mathbb{N}} \textrm{ is fully generic} \right\}.$$
This is a set of universal full measure in $M_1^{\mathbb{N}}$.
\end{theorem}

Lastly, to connect this argument with the return times it will be shown that fully generic sequences of complex values create universal good weights.  His inductive argument is first shown for sequences which are generic for weakly-mixing measures and then lifted to an argument for general shift-invariant measures.  Rather than reproducing his argument here, we include some of the key terminology defined by Rudolph for transforming the return times argument to one about measures and joinings.

\subsection{Definitions}

We will begin by discussing how and why we can transfer the statement of \textbf{Theorem \ref{MultRTT}} to the context of measures on a compact metric space.  In this case we have bounded functions $f_1, \ldots, f_k$ each defined on different measure preserving systems $(X_i, \mathcal{F}_i, \mu_i, T_i)$.  Since the functions are bounded complex-valued functions, we can consider all of them as mapping onto some compact metrizable spaces $M_i$.  For each $f_i$ we consider the map $\Phi_i$ which maps from $X_i$ to $\M_i$ by $\Phi_i(x) = (f_i(x), f_i(T_ix), f_i(T_i^2x), f_i(T_i^3x), \ldots)$.  Thus each $\mu_i$ induces a measure, we'll call $\mu_i^*$ on $\M_i$ where
$\mu_i^*(A) = \mu_i(\Phi^{-1}(A))$.  We can then transfer each one of the systems $(X_i, \mathcal{F}_i, \mu_i, T_i)$ and functions $f_i$ to the system $(\M_i, \mathcal{A_i}, \mu_i^*, \sigma)$ where $\sigma$ is the left shift $\sigma(\vec{m}) = \sigma((m_0, m_1, m_2, \ldots)) = (m_1, m_2, m_3, \ldots)$.   So now each system is now viewed on a compact metrizable space under the same transformation.

If we study the concept of genericity in $\M_i$ for arbitrary shift-invariant measures on $\M_i$ and bounded functions defined on $\M_i$, we will be able to deduce \textbf{Theorem \ref{MultRTT}} in the form in which we are interested.  To see this consider the function $\pi$ defined on $\M_i$ by $\pi(\vec{m}) = \pi((m_0, m_1, m_2, \ldots)) = m_0$, then studying the genericity of a point $\vec{m} = (m_0, m_1, m_2, \ldots)$ with respect to the function $\pi$ means looking at averages of the form
$$\frac{1}{N} \sum_{n=1}^N \pi(\sigma^n (\vec{m})) = \frac{1}{N} \sum_{n=1}^N m_i.$$
So one specific choice for $\vec{m}$ would lead to the averages
$$\frac{1}{N} \sum_{n=1}^N f_i(T_i^nx).$$
Similarly studying the genericity of pairs of points $(\vec{m}_0,\vec{m}_1)$ in some space $(M_0 \times M_1)^{\mathbb{N}}$ leads to the very specific case of averages of the form
$$\frac{1}{N} \sum_{n=1}^N f_0(T_0^nx_0)f_1(T_1^nx_1).$$
Since Rudolph's characterization leads to an inductive argument he will be able to deduce the general result of \textbf{Theorem \ref{MultRTT}} from the case of genericity over two terms as stated in \textbf{Theorem \ref{RT1}}.  From this description it is clear that his result covers a much wider range than we are interested in for the Multiterm Return Times Theorem.  We are just interested in the measures of the form $\mu_i^*$ and points in $\M_i$ which come from mapping by $\phi_i$, but he proves his argument for general measures and points on $\M$.

First, we define some key concepts in terms of measures on a compact metrizable space.  Let $M$ be any compact metrizable space and $\vec{m} = (m_0, m_1, m_2, \ldots)$ be some arbitrary element of $\M$.  We will let $\sigma$ represent the left shift transformation as above.  Define $\mathcal{M}(M)$ as the set of measures on $\M$.  Then $\mathcal{M}_s(M)$ shall represent the shift-invariant measures on $\M$ and $\mathcal{M}_e(M)$ the ergodic measures on $\M$ for $\sigma$ which are the extreme points of $\mathcal{M}_s(M)$.  In this context, we express the genericity as follows.

\begin{defn}
We say $\nu \in \mathcal{M}(M)$ is \textbf{generic for $\mu \in \mathcal{M}_s(M)$} if
$$\frac{1}{N}\sum_{j=0}^{N-1} \sigma_*^j(\nu) \rightarrow \mu.$$
A point $\vec{m} \in M^{\mathbb{N}}$ is called \textbf{generic for a measure $\mu \in \mathcal{M}_s(M)$} if the point mass at $\vec{m}$, $\delta_{\vec{m}}$ is generic for $\mu$.
\end{defn}

Note that this is equivalent to the expressions for genericity discussed above (\textbf{Definitions \ref{FGen}} and \textbf{\ref{BirkGen})} because for any $\mu$-measurable $A \subset \M$ we have
\begin{eqnarray*}
\vec{m} \textrm{ is generic for } \mu & \iff & \delta_{\vec{m}} \textrm{ is generic for } \mu \\
& \iff & \frac{1}{N} \sum_{j=0}^{N-1} \sigma_*^j\left(\delta_{\vec{m}}\right) \rightarrow \mu \\
& \iff & \frac{1}{N} \sum_{j=0}^{N-1} \delta_{\vec{m}}\left(\sigma^{-j}(A)\right) \rightarrow \mu(A) \\
& \iff & \frac{1}{N} \#\left\{j \in [0,N-1] : \vec{m} \in \sigma^{-j}(A)\right\} \rightarrow \mu(A) \\
& \iff & \frac{1}{N} \sum_{j=0}^{N-1} \chi_A(\sigma^j (\vec{m})) \rightarrow \mu(A) = \int \chi_A d\mu
\end{eqnarray*}

\begin{defn}
For a measure $\nu$ to be \textbf{pointwise generic} for a shift-invariant measure $\mu$ requires that for $\nu$-a.e. $\vec{m}$, the point $\vec{m}$ is generic  for some ergodic measure $\mu(\vec{m})$ with the property
$$\int \mu(\vec{m}) d\nu = \mu.$$
\end{defn}

It follows from \textbf{Theorem \ref{Birk}} that any shift-invariant measure is pointwise generic with respect to itself.  This is because Birkhoff's Pointwise Ergodic Theorem implies that for any $\mu \in \mathcal{M}_s(M)$, then for $\mu$-a.e. $\vec{m} \in \M$ and any $\mu$-measurable $A \subset \M$ we have
$$\frac{1}{N} \sum_{j=0}^{N-1} \chi_A(\sigma^j(\vec{m})) \rightarrow \E\left(\chi_A|\textit{Inv}\right)(\vec{m})$$
and
$$\int \E\left(\chi_A|\textit{Inv}\right)(\vec{m}) d\mu = \int \chi_A d\mu = \mu(A).$$

\begin{defn}
We say $\vec{m}$ is \textbf{fully generic for $\mu \in \mathcal{M}_e(M)$} if $\delta_{\vec{m}}\times \mu^{\mathbb{N}}$ is pointwise generic for $\mu \times \mu^{\mathbb{N}}$.  A point $\vec{m} \in M^{\mathbb{N}}$ is a \textbf{fully generic point} if it is generic for some measure $\mu \in \mathcal{M}_e(M)$ and is fully generic for $\mu$.
\end{defn}

Thus for $\delta_{\vec{m}}\times \mu^{\mathbb{N}}$-a.e. sequences of points $(\vec{m}_0, \vec{m}_1, \vec{m}_2, \ldots)$, the sequence $(\vec{m}_0, \vec{m}_1, \vec{m}_2, \ldots)$ is generic for some ergodic measure $\mu((\vec{m}_0, \vec{m}_1, \vec{m}_2, \ldots))$ with the property that
$$\int \mu((\vec{m}_0, \vec{m}_1, \vec{m}_2, \ldots)) d(\delta_{\vec{m}} \times \mu^{\mathbb{N}}) = \mu \times \mu^{\mathbb{N}}$$
So for $\mu^{\mathbb{N}}$-a.e.  $(\vec{m}_1, \vec{m}_2, \ldots)$, the sequence $(\vec{m}, \vec{m}_1, \vec{m}_2, \ldots)$ is generic for some ergodic measure $\mu((\vec{m}_1,\vec{m}_2, \ldots))$ with the property that
$$\int \mu((\vec{m}_1, \vec{m}_2, \ldots)) d\mu^{\mathbb{N}} = \mu \times \mu^{\mathbb{N}}.$$
For every $\mu \times \mu^{\mathbb{N}}$-measurable set $\overline{A}$ and $\mu^{\mathbb{N}}$-a.e.  $(\vec{m}_1, \vec{m}_2, \ldots)$ we have
\begin{eqnarray*}
\left(\mu \times \mu^{\mathbb{N}}\right)(\overline{A}) & = & \int \chi_{\overline{A}} d\left(\mu \times \mu^{\mathbb{N}}\right) \\
& = &\int \int \chi_{\overline{A}} d\mu((\vec{m}_1, \vec{m}_2, \ldots)) d\mu^{\mathbb{N}} \\
& = & \int \lim \frac{1}{N} \sum_{j=0}^{N-1} \chi_{\overline{A}}((\sigma \times \sigma \times \sigma \times \cdots)^j\left((\vec{m}, \vec{m}_1, \vec{m}_2, \ldots)\right) d\mu^{\mathbb{N}} \\
& = & \int \lim \frac{1}{N} \sum_{j=0}^{N-1} \chi_{\overline{A}}\left(\sigma^j(\vec{m}), \sigma^j(\vec{m}_1), \sigma^j(\vec{m}_2), \ldots)\right) d\mu^{\mathbb{N}}
\end{eqnarray*}

Let $\mu \in \mathcal{M}_e(M)$.  If $A$ is a $\mu$-measurable set, then the set $\overline{A} = A \times \{\} \times \{\} \times \cdots$ is a $\mu \times \mu^{\mathbb{N}}$-measurable set.  If $\vec{m}$ is fully generic for $\mu$, then by the above analysis
\begin{eqnarray*}
\mu(A) & = & \left(\mu \times \mu^{\mathbb{N}}\right)(\overline{A}) \\
& = & \int \lim \frac{1}{N} \sum_{j=0}^{N-1} \chi_{\overline{A}}\left(\sigma^j(\vec{m}), \sigma^j(\vec{m}_1), \sigma^j(\vec{m}_2), \ldots)\right) d\mu^{\mathbb{N}} \\
& = & \int \lim \frac{1}{N} \sum_{j=0}^{N-1} \chi_A(\sigma^j (\vec{m})) d\mu^{\mathbb{N}} \\
& = & \lim \frac{1}{N} \sum_{j=0}^{N-1} \chi_A(\sigma^j (\vec{m}))
\end{eqnarray*}
Thus $\vec{m}$ is generic for $\mu$ as well.

\begin{defn}
A set $A \subseteq M^{\mathbb{N}}$ is of \textbf{universal full measure} if for all $\mu \in \mathcal{M}_s(M)$, $\mu (A) = 1$.  It is enough to check this just for the ergodic measures.
\end{defn}

The companion statement of \textbf{Theorem \ref{Birk}} in this new context is given below.

\begin{lemma}\label{Birk2}
The set of points $\vec{m} \in M^{\mathbb{N}}$ that are generic for an ergodic measure $\mu$ is a set of universal full measure.  That is to say, for any $\mu \in \mathcal{M}_e(M)$
$$\mu(\{\vec{m}:\vec{m} \textrm{ is generic for } \mu \}) = 1.$$
\end{lemma}

Using the above definitions, we wish to establish the following result concerning the set of fully generic points.

\begin{theorem}
Let $A = \{\vec{m} \in M^{\mathbb{N}} : \vec{m} \textrm{ is a fully generic point } \}$.  Then the set $A$ has universal full measure.
\end{theorem}

Let $\mu \in \mathcal{M}_e(M)$.  Let $A_{\mu} = \{\vec{m} \in M^{\mathbb{N}} : \vec{m} \textrm{ is a fully generic point for } \mu \}$.  Then
$$A = \bigcup_{\mu \in \mathcal{M}_e(M)} A_{\mu}.$$
To show that $A$ is of universal full measure, I need only show that $\mu(A_{\mu}) = 1$ for all $\mu \in \mathcal{M}_e(M)$.  As $\mu \times \mu^{\mathbb{N}}$ is a shift-invariant measure, it is pointwise generic with respect to itself.  Thus the set of all points in $\vec{\vec{m}} = (\vec{m}_0, \vec{m}_1, \vec{m}_2, \ldots)$ in $\M \times (\M)^{\mathbb{N}}$ which are generic for some ergodic measure $\mu(\vec{\vec{m}})$ has full measure with respect to $\mu \times \mu^{\mathbb{N}}$.  Call this set $B_{\mu,1}$.  By \textbf{Lemma \ref{Birk2}} the set of all points $\vec{m}_0$ such that $\vec{m}_0$ is generic for $\mu$ has full measure with respect to $\mu$.  Call this set $G_{\mu}$.  Define $B_{\mu,2} = G_{\mu} \times \M \times \M \times \M \times \cdots$.  Then $B_{\mu,2}$ has full measure with respect to $\mu \times \mu^{\mathbb{N}}$.  Let $B_{\mu} = B_{mu,1} \cap B_{\mu,2}$.  Then $B_{\mu}$ has full measure with respect to $\mu \times \mu^{\mathbb{N}}$.  If $(\vec{m}, \vec{m}_1, \vec{m}_2, \ldots) \in B_{mu}$, then $\vec{m} \in A_{\mu}$ and thus $\mu(A_{\mu}) = 1$.

\section{Characteristic Factors}\label{Character}

As you can see in the BFKO proof of the Return Times Theorem one of the keys to the argument was to break up the function using the Kronecker factor in order to prove the result independently for both the eigenfunctions and those functions in the orthocomplement of the Kronecker factor.  Using factors in proving convergence in ergodic theory has long been a very useful tool.  The notion of a characteristic factor is originally due to H. Furstenberg and is explicitly defined in \cite{FW}.

\begin{defn}
When the limiting behavior of a non-conventional ergodic average for $(X,\mathcal{F}, \mu, T)$ can be reduced to that of a factor system $(Y, \mathcal{G}, \nu, T)$, we shall say that the latter is a \textbf{characteristic factor} of the former.
\end{defn}

For each type of average under consideration, one will have to specify what is meant by reduced in the given case.  One approach to studying the convergence of nonconventional averages is to find the \textbf{minimal characteristic factor} which is the smallest factor which is characteristic for a given type of recurrence.  In the case of Furstenberg and Weiss \cite{FW}, they define the notion of characteristic factor when finding a characteristic factor for averages of the type
$$\frac{1}{N}\sum_{n=1}^N f \circ T^n g\circ T^{n^2}.$$
Therefore their specific definition of characteristic factor is as follows.

\begin{defn}
If $\{p_1(n), p_2(n), \ldots, p_k(n)\}$ are $k$ integer-valued sequences, and $(Y,\mathcal{G}, \nu, T)$ is a factor of a system $(X, \mathcal{F}, \mu, T)$, we say that $\mathcal{G}$ is a \textbf{characteristic factor for the scheme} $\{p_1(n), p_2(n), \ldots, p_k(n)\}$, if for any $f_1, f_2, \ldots, f_k \in L^{\infty}(\mu)$ we have
$$\frac{1}{N} \sum_{n=1}^N (f_1 \circ T^{p_1(n)})\cdots (f_k \circ T^{p_k(n)})-\frac{1}{N} \sum_{n=1}^N (\E(f_1 | \mathcal{G}) \circ T^{p_1(n)})\cdots (\E(f_k,\mathcal{G} \circ T^{p_k(n)})$$
converges to $0$ in $L^2(\mu)$.
\end{defn}

\subsection{Characteristic Factors and the Return Times Theorem}

The proof by D. Rudolph of the Multiterm Return Times Theorem (\textbf{Theorem \ref{MultRTT}}) gives an elegant proof of the theorem, but avoids the creation of the null set off of which the averages converge and the discovery of the factor which was characteristic for the return times averages.  Both the null set and the characteristic factor played a key role in the proofs of \textbf{Theorem \ref{RTT}} for two terms in \cite{RTT2} and \cite{RudolphAJP}, thus there is some interest in identifying what factors are characteristic for the Multiterm Return Times Theorem (\textbf{Theorem \ref{MultRTT}}).

In 2012, I. Assani and K. Presser published an update \cite{AP2} of their earlier unpublished work \cite{AP} on characteristic factors and the Multiterm Return Times Theorem.  In a reversal of the BFKO argument which used the properties of the characteristic factor to help prove the Return Times Theorem, Assani and Presser use the convergence of the multiterm return times averages guaranteed by Rudolph's Multiterm Return Times Theorem \cite{RudolphFGS} to demonstrate the convergence properties for two types of factors.

We first consider the factors used by H. Furstenberg to prove Szemer\'{e}di's Theorem \cite{FurstEBO}.  These factors are called $k$ step distal factors in \cite{FurstEBO}. We denote these factors as $\mathcal{A}_k$ using the notation from \cite{AssaniCFF} where these factors were shown to be $L^2$-characteristic for the averages
$$\frac{1}{N}\sum_{n=1}^N \prod_{i=1}^I f_i\circ T^{in}.$$

\begin{defn} Let $(X, \mathcal{F}, \mu, T)$ be an ergodic dynamical system on a probability measure space.  The factors $\mathcal{A}_k$ are defined in the following inductive way.
\begin{itemize}
\item The factor $\mathcal{A}_0$ is equal to the trivial $\sigma$-algebra $\{X,\emptyset\}$
\item For $k\geq 0$ the factor $\mathcal{A}_{k+1}$ is characterized by the following.  A function $f\in \mathcal{A}_{k+1}^{\perp}$ if and only if
$$N_{k+1}(f)^4 := \lim_H\frac{1}{H}\sum_{h=1}^H \left\|\mathbb{E}(f\cdot f\circ T^h | \mathcal{A}_{k})\right\|_2^2 = 0$$
\end{itemize}
\end{defn}

In the aforementioned paper, we prove that these seminorms are well-defined and characterize factors which are successive maximal isometric extensions.

\begin{theorem}\label{A_k}
Let $k$ be any positive integer.  For any ergodic dynamical system $(X,\mathcal{F}, \mu, T)$ and for each $f \in L^{\infty}(\mu)$ we can find a set of full measure $X_f$ such that for each $x \in X_f$, for any other dynamical system $(Y_1,\mathcal{G}_1, S_1, \nu_1)$ and any $g_1 \in L^{\infty}(\nu_1)$ with $\|g_1\|_{\infty} \leq 1$, there exists a set of full measure $Y_{g_1}$ such that for each $y_1$ in $Y_{g_1}$ then $\ldots$ for any other dynamical system $(Y_{k-1},\mathcal{G}_{k-1}, S_{k-1}, \nu_{k-1})$ and any $g_{k-1} \in L^{\infty}(\nu_{k-1})$ with $\|g_k\|_{\infty} \leq 1$ there exist a set of full measure $Y_{g_{k-1}}$ in $Y_{k-1}$ such that if $y_{k-1} \in Y_{g_{k-1}}$ for any other dynamical system $(Y_k,\mathcal{G}_k, S_k, \nu_k)$ for $\nu_k$-a.e. $y_k$
\begin{itemize}
    \item the average
        \begin{equation}\label{RTTEq}
            \frac{1}{N}\sum_{n=1}^N \left[f(T^{n}x)-\mathbb{E}(f | \mathcal{A}_k)(T^nx)\right]g_{1}(S_{1}^{n}y_{1})g_{2}(S_{2}^{n}y_{2})\cdots g_{k}(S_{k}^{n}y_{k})
        \end{equation}
        converges to 0.
    \item Thus for $f\in \mathcal{A}_k^{\perp}$ the average
$$\frac{1}{N}\sum_{n=1}^N f(T^{n}x)g_{1}(S_{1}^{n}y_{1})g_{2}(S_{2}^{n}y_{2})\cdots g_{k}(S_{k}^{n}y_{k})$$
converges to 0 $\nu_k$-a.e..
    \item Also we have the following pointwise upper bound for our limit
        \begin{equation}\label{Upbound}
            \limsup_N\left|\frac{1}{N}\sum_{n=1}^N f(T^nx)g_1(S_1^ny)g_{2}(S_{2}^{n}y_{2})\cdots g_{k}(S_{k}^{n}y_{k})\right|^2\leq CN_{k+1}(f)^2
        \end{equation}
\end{itemize}
\end{theorem}

An important aspect of the equation (\ref{Upbound}) is the control of the orbits made by the seminorms $N_{k+1}$.

The study of the nonconventional Furstenberg averages has seen important progress being made in the last seven years. In \cite{HK-NEA} and \cite{Ziegler} the Host-Kra-Ziegler factors $\mathcal{Z}_k$ were created independently by B. Host, B. Kra and T. Ziegler and were shown to be characteristic in $L^2$ norm for the Furstenberg averages.

\begin{defn}  Let $(X, \mathcal{F}, \mu, T)$ be an ergodic dynamical system on a probability measure space.  The factors $\mathcal{Z}_k$ are defined as follows.
\begin{itemize}
    \item The factor $\mathcal{Z}_0$ is equal to the trivial $\sigma$-algebra.
    \item The factor $\mathcal{Z}_1$ can be characterized by the seminorms $\||f|\|_2$  where
        $$\||f|\|_2^4 = \lim_{H}\frac{1}{H}\sum_{h=1}^{H} \left|\int f\cdot f\circ T^hd\mu\right|^2 $$
    \item The factor $\mathcal{Z}_2$ is the Conze-Lesigne factor, $\mathcal{CL}$.  Functions in this factor are characterized by the seminorm $|\|\cdot |\|_3$ such that
        $$\||f|\|_3^8 = \lim_{H}\frac{1}{H} \sum_{h=1}^{H}\||f\cdot f\circ T^h|\|_2^4.$$
        A function $f\in \mathcal{CL}^{\perp}$ if and only $\||f|\|_3 =0.$
    \item More generally B. Host and B. Kra showed in \cite{HK-NEA} that for each positive integer $k$ we have
        \begin{equation}\label{Semi}
            \||f|\|_{k+1}^{2^{k+1}}= \lim_H\frac{1}{H}\sum_{h=1}^{H}\||f\cdot f\circ T^h|\|_{k}^{2^k},
        \end{equation}
        with the condition that $f\in \mathcal{Z}_{k-1}^{\perp}$ if and only if $\||f|\|_{k}=0.$
\end{itemize}
\end{defn}

\begin{theorem}\label{HKZ}
Let $(X, \mathcal{F}, \mu, T)$ be an ergodic measure preserving system.  The Host-Kra-Ziegler factors $\mathcal{Z}_k$ are pointwise characteristic for the multiterm return times averages.
\end{theorem}

As the $\mathcal{Z}_k$ factors are smaller than the factors $\mathcal{A}_k$, and thus $\mathcal{A}_k^{\perp} \subseteq \mathcal{Z}_k^{\perp}$, the fact that the $\mathcal{A}_k$ factors are pointwise characteristic for the multiterm return times averages is a consequence of \textbf{Theorem \ref{HKZ}}. But in our proof of \textbf{Theorem \ref{A_k}} using the seminorm defining the factors $\mathcal{A}_k$ we obtain pointwise uniform upper bounds of the multiterm return times averages. With the $\mathcal{Z}_k$ factors we do not have such pointwise estimates. The uniform upper bounds are derived after integration combined with a $\limsup$ argument See \cite{AP2} and Example 5.1 in T. Eisner and P. Zorin-Kranich \cite{E12}.

\section{Breaking the Duality}\label{Duality}

As mentioned above the original Return Times Theorem (\textbf{Theorem \ref{RTT}}) can be easily extended to $L^p(\mu)$ and $L^q(\nu)$ where $\frac{1}{p}+\frac{1}{q} \leq 1$ using the Banach Principle and H\"older's inequality.  What can be said about convergence of return times averages when $\frac{1}{p}+\frac{1}{q} > 1$?

A first motivation for this question comes from the following result due to I. Assani (1990, see page 141 in \cite{AssaniAWP}).

\begin{theorem}
Let $(X, \mathcal{F}, \mu, T)$ be a measure preserving system and $f\in L^1(\mu)$ then for $\mu$-a.e. $x$, for each measure preserving system $(Y, \mathcal{G}, \nu, S)$ and every $g\in L^1(\nu)$ the averages
$$\frac{1}{N}\sum_{n=1}^N f(T^nx) g\circ S^n$$
converge in $L^1(\nu)$ norm.
\end{theorem}

The proof is a consequence of the Wiener-Wintner Ergodic Theorem (\textbf{Theorem \ref{WW}}), the Spectral Theorem and the Maximal Ergodic Theorem \cite{AssaniWWE}.  In view of this result one could look for the true nature of the dynamic involved in the return times theorem allowing to go beyond the H\"{o}lderian duality.

A first evidence of the possible validity of the return times theorem beyond the H\"{o}lderian duality was shown in \cite{AssaniSLF} and \cite{AssaniAWP}.  In \cite{AssaniSLF}, I. Assani showed that if $(X, \mathcal{F}, \mu, T)$ is a measure preserving system and $f\in L^p$, $1<p\leq \infty$, then we can find a set $X_f$ of full measure such that for every sequence of independent, identically distributed (i.i.d) random variables $X_n$ in $L^1$ we have
$$\lim_n \frac{f(T^nx)X_n(\omega)}{n}=0$$
for a.e. $\omega.$  This leaves the case where $p=1$ which is addressed below in \textbf{Subsection \ref{L1L1}} on the $(L^1, L^1)$ case.
Then in \cite{AssaniAWP}, I. Assani showed that a sequence $\{X_n\}$ of i.i.d. random variables defined on the probability space $\Omega$ and having a finite $p$-th moment for some $1<p < \infty$ are universal good weights for a.e. convergence of the return times averages on $L^q$ where $1 < q < \infty$.  That is to say that the return times holds for the pair $(L^p, L^q)$ when $1< p,q < \infty$ when we restrict the initial term to i.i.d. random variables.

Subsequent work by C. Demeter in \cite{DemeterTBC} looked at breaking the duality in weighted ergodic averages of the form
$$\frac{1}{N}\sum_{n=1}^Na(k)f(T^kx)$$
where $a(k)$ is a sequence of complex numbers and $T$ is a linear operator of some $L^p$ space with $1 \leq p \leq \infty$ and $f \in L^p$.

Further progress along these lines can be seen in Theorem 1.6 of C. Demeter, M. Lacey, T. Tao and C. Thiele \cite{DLTT-BTD}, see also \cite{AssaniSLF} and \cite{BJLO}, which is stated below.

\begin{theorem}\label{DLTT1-6}
Assume that either $p > 1$ and $q = 1$, or $p = 1$ and $q > 1$. For each dynamical system
$(X,\mathcal{F}, \mu, T)$ and each $f \in L^p(\mu)$ there is a set $X^* \subseteq X$ of full measure, such that for each sequence of $L^p$ i.i.d. random variables $Y_n$ defined on the probability space $(Y,\mathcal{G},\nu)$ and each $x\in X^*$,
$$\lim_N\frac{1}{N} \sum_{n=1}^N f(T^nx)Y_n(y)$$
exists for $\nu$-a.e $y.$
\end{theorem}

\begin{defn}
A sequence of complex numbers $\{a(k)\}$ is $p$-Besicovitch if $\{a(k)\}$ is in the closure of the trigonometric polynomials in the semi-norm
$$\left(\limsup_N\frac{1}{N} \sum_{k=1}^N|a(k)|^p\right)^{\frac{1}{p}}.$$
\end{defn}

In \cite{LOT}, M. Lin, J. Olsen and A. Templeman showed that if $a(k)$ is $q$-Besicovitch and $f \in L^p$ where $\frac{1}{p}+\frac{1}{q}=1$ then the above averages converge for all $f \in L^p$ when $T$ is a Dunford-Schwartz operator.  J. Baxter, R. Jones, M. Lin and J. Olsen \cite{BJLO} present a construction which demonstrates that the duality is necessary in the case that $q = 1$.  They show that if $\{X_n\}$ is a nonnegative i.i.d. sequence with $E(|X_1|) < \infty$ and $E(X_1) = 0$ which is not essentially bounded then there exists a sequence $\{a(n)\}$ of nonnegative numbers such that the sequence of its arithmetic means converges to zero and the sequence of weighted averages of the $X_n$ taken with the weights $a(n)$ is as bad as possible, i.e. the $\liminf$ heads to $0$ while the $\limsup$ heads to $\infty$ (see also \cite{CL09,LW}).  In \cite{DemeterTBC}, C. Demeter showed  that the duality is necessary for the $q$-Besicovitch sequences to make universal good weights when $q > 1$ as well.  This result of Demeter is simplified in the paper \cite{DJ} where C. Demeter and R. Jones describe a possible approach to dealing with duality with respect to the return times.

\subsection{Hilbert Transforms}

While the Return Times Theorem (\textbf{Theorem \ref{RTT}}) looks at the convergence of weighted Ces\'aro averages, these averages are closely related to the discrete ergodic Hilbert transform
$$\lim_n\sum_{k=-n}^n\frac{f(T^kx)g(S^ky)}{k}, k \neq 0$$
which is a tool of Harmonic Analysis which was first studied by M. Cotlar \cite{Cotlar}, see also A. P. Calderon \cite{Cald}.

The connection between the convergence of the Cesaro averages and the existence of the ergodic Hilbert transform has been established by R. Jajte \cite{Jajte} for $L^2$ functions.  For the one-sided ergodic Hilbert transform a simple partial summation argument shows that the convergence of the series $\sum_{n=1}^{\infty} \frac{a_n}{n}$ implies the convergence of the averages $\frac{1}{N}\sum_{n=1}^N a_n.$

\begin{conj}\label{HT}
Given any dynamical system $(X, \mathcal{F}, \mu, T)$ and $f \in L^p(\mu)$, there exists a set of full measure $X_f \in X$ such that for all $x \in X_f$ and for every other dynamical system $(Y, \mathcal{G}, \nu, S)$ and $g \in L^q(\nu)$, the limit
$$\lim_n\sum_{k=-n}^n\frac{f(T^kx)g(S^ky)}{k}, k \neq 0$$
exists for $\nu$-a.e. $y$.
\end{conj}

As referenced by \cite{Tal}, Lacey and Marcus have shown that the conjecture is false for $p=1$ using the family of rotations on the torus as the initial weight and applying that to $g(y) = e^{2\pi i y}$.  They actually proved the existence of an $L^1$ sequence $\{X_k\}$ of i.i.d. random variables such that for almost every $x$
$$\lim_n\sum_{k=-n}^n \frac{X_k(x)e^{2 \pi i k\theta}}{k}, k \neq 0$$
fails to exists for some $\theta$.  Similarly, Talagrand showed the same negative result holds for the one-sided version of the series in the course of proving the following theorem.

\begin{theorem}
Given an i.i.d. symmetric sequence of random variables $\{X_k\}$, the random Fourier series
$$\sum_{k=1}^{\infty} \frac{X_ke^{2 \pi i k \theta}}{k}$$
converges uniformly for a.e. $\theta$ if and only if $\{X_k\}$ is in $L \log \log L$.
\end{theorem}

However, Cuzick and Lai \cite{CuL} proved that for each $L^p$ sequence $\{X_k\}$ of i.i.d. random variables where $p > 1$ there exists a set $X^* \in X$ with the property that
$$\lim_n\sum_{k=-n}^n \frac{X_k(x)e^{2 \pi i k\theta}}{k}$$
exists for each $x \in X^*$ and $\theta \in \mathbb{T}$.

To study duality and the weighted convergence with respect to the Hilbert transform, it has become a customary strategy following in the models as above to look at the role that i.i.d. sequences play in the convergence of such averages.  In \cite{AssaniDAT}, the first author proved the following about duality and the Hilbert transform.

\begin{theorem}
Let $(X, \mathcal{F}, \mu, T)$ be a dynamical system and $f \in L^1(\mu)$.  There is a set $X_f \in X$ of full measure such that for each sequence $\{Y_k\}$ of i.i.d. random variables defined on the probability space $(Y, \mathcal{G}, \nu)$, with $Y_1 \in L^q(X), q > 1$ and each $x \in X_f$
$$\lim_n \sum_{k=-n}^n \frac{f(T^kx)Y_k(y)}{k}, k \neq 0$$
exists for $\nu$-a.e. $y$.  The result fails when $q = 1$, for some $f$ in every ergodic dynamical system $(X, \mathcal{F}, \mu,T)$.
\end{theorem}

A companion result was proven in \cite{AssaniDAT} this time with the i.i.d. sequence playing the role of the weight in the average.

\begin{theorem}
Let $\{X_k\}$ be a mean $0$ sequence of i.i.d. random variables defined on the probability space $(X, \mathcal{F}, \mu)$ which are assumed to be in $L^p$ for some $p$ with $1 < p \leq \infty$.  Then there exists a subset $X* \in X$ of full measure such that for each $x \in X*$, the following holds: for any dynamical system $(Y, \mathcal{G}, \nu, S)$ and $g \in L^r$ with $1 < r \leq \infty$
$$\lim_n \sum_{k=-n}^n \frac{X_k(x)g(S^ky)}{k}, k \neq 0$$
exists for $\nu$-a.e. $y$.
\end{theorem}

Cuny \cite{Cuny} showed that this last result could be extended to the case where $X_1 \in L \log L$ and $g\in L \log L$.  Demeter \cite{DemeterRWS} showed that for the series
$$\lim_n\sum_{k=-n}^n \frac{X_k(x)Y_k(y)}{k}, k \neq 0$$
one can obtain a.e. convergence if both $X_n$ and $Y_n$ are i.i.d. sequences in $L^1.$

In \cite{DLTT-BTD}, C. Demeter, M. Lacey, T. Tao and C. Thiele showed that \textbf{Conjecture \ref{HT}} is true if $1 <p \leq \infty$ and  $q \geq 2$.  One of their main results is the following theorem.

\begin{theorem}\label{DLTT2}
Let $K: \mathbb{R} \rightarrow \mathbb{R}$ be an $L^2$-kernel satisfying the following requirements:
\begin{equation}
\widehat{K} \in C^{\infty}(\mathbb{R} \setminus \{0\}),
\end{equation}
\begin{equation}
|\widehat{K}(\xi)| \lesssim \min\left\{1,\frac{1}{|\xi|}\right\}, \forall \xi \neq 0,
\end{equation}
\begin{equation}
\left|\frac{d^n}{d\xi^n} \widehat{K}(\xi)\right| \lesssim \frac{1}{|\xi|^n}\min\left\{|\xi|,\frac{1}{|\xi|}\right\}, \forall \xi \neq 0, n \geq 1.
\end{equation}
Then the following inequality holds for each $1 < p \leq \infty$.  There exists a finite constant $C_p$ such that
$$\left\|\sup_{\|g\|_{L^2(\mathbb{R})}=1}\left\|\sup_{k \in \mathbb{Z}}\left|\frac{1}{2^k} \int f(x+y)g(z+y)K\left(\frac{y}{2^k}\right)dy\right| \right\|_{L_z^2(\mathbb{R})}\right\|_{L_x^p(\mathbb{R})} \leq C_p \|f\|_{L^p(\mathbb{R})}$$
where $$\|F\|_{L^2_z}(\mathbb{R}) = \bigg(\int_{-\infty}^{\infty} |F(z)|^2 dz\bigg)^{1/2}.$$
\end{theorem}

Later in their paper (Corollary 3.8), they transfer this result to the ergodic setting and show that the associated return times theorem holds  when $1<p \leq \infty$ and  $q \geq 2$.
This breaks the duality but leaves open the remaining cases  where $\frac{1}{p} + \frac{1}{q} < 2.$ In particular, if one looks at \textbf{Theorem \ref{DLTT1-6}}, one could reasonably ask:

\begin{ques}
Does the Return Times theorem hold for $p=1$ and $q>1$ or $p>1$ and $q=1$?
\end{ques}

Because of the failure of the convergence of the series
$$\sum_{n=-\infty}^{\infty} \frac{f(T^nx)g(S^ny)}{n}$$
for $p=1$ and $q = \infty$ (by the Lacey, Marcus result (referenced in \cite{Tal}) one can not expect the use of the Hilbert transform series to answer the above question.  One can observe that while the Cesaro averages of nonnegative functions is nonnegative, the Hilbert transform does not have this property.  In \cite{DemeterOSM}, however, it is announced that the range of validity of the return times could be extended to $\frac{1}{p} + \frac{1}{q} <3/2$.  See also \cite{OSTTW}.

A second corollary of \textbf{Theorem \ref{DLTT2}} is Theorem 3.4 of the Demeter, Lacey, Tao and Thiele paper which only shows that the set of convergence is closed. Having no obvious dense set the authors refine their techniques to prove a variational inequality which establishes the pointwise convergence of the ergodic Hilbert transform for the specific range of $p$ and $q$.

The method used in \cite{DLTT-BTD} adapts the tools developed by M. Lacey and C. Thiele in their understanding and applications of the celebrated Carleson-Hunt theorem on the convergence of the Fourier series of $L^p$ functions and the bilinear Hilbert transform and Calderon conjecture. The upper bound of $3/2$ appears in their papers.

One wonders if one needs to use such powerful tools to solve the problem of the break of duality. In other words is the difficulty of breaking the duality in the return times theorem at the level of the Carleson-Hunt theorem? This seems to be the case as indicated by T. Tao in one of his blogs.\footnote{http://terrytao.wordpress.com/2007/12/11/}  As such, it would appear that the problem of the break of duality for the return times theorem could shed new light on the Carleson-Hunt theorem and produce substantial refinements. A first step in this direction is a variational Carleson Hunt theorem obtained in \cite{OSTTW}.  Indeed, based on the Multiple Return Times Theorem established by D. Rudolph one can ask the following question.

\begin{ques}
What is the range of values $(p_1, p_2, p_3,...,p_H)$ for which the $H$ term return times theorem holds?
\end{ques}

\subsection{The $(L^1, L^1)$ Case}\label{L1L1}

Does the return times theorem hold for the pair $(L^1, L^1)$?  In 1990, Assani \cite{AssaniWWD} proved that for a finite measure preserving system $(X, \mathcal{F}, \mu, T)$ and $f \in L^1(\mu)$ then for $\mu$-a.e. $x \in X$ the sequence $\{f(T^nx)\}$ is a good universal weight for the norm convergence in $L^1$.

This initial result gives some support for the possibility that the return times theorem would hold for the pair $(L^1, L^1)$.  To approach the question of the return time for $(L^1, L^1)$, it was suggested in \cite{AssaniSLF} to look at the return times for the tail.

\begin{defn}
The \textbf{return times property holds for the tail for the pair $(L^s, L^t)$} if for all $f \in L^s(\mu)$ we can find a set of full measure $X_f$ such that for each $x \in X_f$ for all measure preserving systems $(Y, \mathcal{G}, \nu, S)$ and for all $g \in L^t(\nu)$ the sequence
$$\frac{f(T^nx)g(S^ny)}{n}$$
converges $\nu$-a.e. to 0.
\end{defn}

In \cite{AssaniSLF}, the first author showed that the validity of the return times for the tail for the pair $(L^1, L^1)$ is equivalent to the following counting problem.

\begin{conj}
For any measure preserving system $(X, \mathcal{F}, \mu, T)$ and for all $f \in L^1(\mu)$ and $\mu$-a.e. $x$
$$N^*(f)(x) = \sup_n \frac{\#\left\{k \in \mathbb{N} : \frac{|f|(T^kx)}{k} \geq \frac{1}{n}\right\}}{n} < \infty$$
\end{conj}

In \cite{AssaniPSeries}, I. Assani showed that $N^*(f) \in L^1$ if $f \in L\log L$.  In \cite{DQ}, C. Demeter and A. Quas showed that $N^*(f)(x) < \infty$ a.e. when $f\in L \log \log L$. However, in the papers \cite{ABM-ALC}, this conjecture is proven false and its connection to return times is discussed in greater detail.  The key result is Theorem 1 of \cite{ABM-ALC} which states

\begin{theorem}\label{FailRTT}
In any nonatomic, invertible ergodic system $(X, \mathcal{F}, \mu, T)$ there exists $f \in L_+^1$ such that
$$\sup_n \frac{N_n(f)(x)}{n} = \infty$$
almost everywhere where
$$N_n(f)(x) = \#\left\{k : \frac{f(T^kx)}{k} > \frac{1}{n}\right\}.$$
\end{theorem}

The work in \cite{ABM-ALC} demonstrates that with the above result the return times for the tail does not hold for the pair $(L^1, L^1)$ and thus the return times theorem itself does not hold for the pair $(L^1, L^1)$.  As noted in \cite{DQ}, the method used in \cite{ABM-ALC} to prove \textbf{Theorem \ref{FailRTT}} shows in fact that for $f$ in any Orlicz space strictly bigger than $L \log \log \log L$ we can still have $N^*(f)(x) = \infty$ a.e.

\begin{ques}
Is $N^*(f)(x)<\infty$ a.e. for $f\in L \log \log \log L$?
\end{ques}

A related question was raised in \cite{CT} by M. J Carro and P. Tradacete for the following related operator $A$ introduced by I. Assani (see \cite{ABM-ALC}) and defined pointwise as
$$A(f)(x) = \sup_{\lambda}\lambda\cdot m\left\{0 < y < x: \frac{f(x - y)}{y} > \lambda\right\}.$$

In Section 5.1 of \cite{ABM-CAC} we see some applications for this negative result for the Return Times in $L^1$.  Assume that $(\Omega, \mathcal{B}, P)$ is a probability measure space and $Y_1, Y_2, \ldots$ is a sequence of i.i.d. random variables of values in ${-1,1}$ with $P(Y_n = 1)= \sigma$ and $P(Y_n=-1) = 1-\sigma$ where $\frac{1}{2} \leq \sigma \leq 1$.  Set
$$a_n(\omega)= \sum_{k=1}^n Y_k(\omega).$$
By the strong law of large numbers we know that
$$\lim_n \frac{a_n(\omega)}{n} = E(Y_1) = 2\sigma -1.$$
If $\sigma > \frac{1}{2}$ the for $\mu$-a.e. $\omega$ we have $\lim_n a_n(\omega) = \infty$.  Fix such an $\omega$ and let $f \in L^p(\mu)$ where $1 \leq p \leq \infty$.  It was proven in \cite{LPWR} that if $p>1$ and $\sigma>\frac{1}{2}$, then the averages
$$\frac{1}{N}\sum_{n=1}^Nf(T^{a_n(\omega)}x)$$
converge for $\mu$-a.e. $x$.  Using \textbf{Theorem \ref{FailRTT}} we have the following result.

\begin{theorem}
Consider a sequence of i.i.d. random variables $Y_1, Y_2, ldots$ defined on a probability measure space $(\Omega, \mathcal{B}, P)$ of values ${-1,1}$.  Assume that $P(Y_n=1) = \sigma$ and $P(Y_n = -1) = 1-\sigma$ with $\sigma > \frac{1}{2}$.  Set $a_n(\omega) = \sum_{k=1}^nY_k(\omega)$ and fix $\omega$ such that $\lim_n a_n(\omega)= \infty$.  In any aperiodic dynamical system there exists a function $f \in L^1(\mu)$ such that the averages
$$\frac{1}{N}\sum_{n=1}^Nf(T^{a_n(\omega)}x)$$
do not converge almost everywhere.
\end{theorem}

\section{Other Notes on the Return Times Theorem}

\subsection{The Sigma-Finite Case}

As Birkhoff's Pointwise Ergodic Theorem (\textbf{Theorem \ref{Birk}}) actually holds for $\sigma$-finite measure spaces, one question to consider is whether the Return Times Theorem (\textbf{Theorem \ref{RTT}}) can be extended to $\sigma$-finite measure spaces as well.  This question was address by Assani in \cite{AssaniTRT} in the following theorem.

\begin{theorem}
Let $(X, \mathcal{F}, \mu, T)$ be a measure preserving system on the $\sigma$-finite measure space $(X, \mathcal{F}, \mu)$.  Given a set $A$ with finite measure, then the sequence $\{\chi_A(T^nx)\}$ is $\mu$-a.e., a good universal weight for the pointwise ergodic theorem in $L^1$.
\end{theorem}

The proof utilizes the method of Hopf's decomposition \cite{KrengelET} to extend the BFKO result to a $\sigma$-finite measure space.

One cannot extend the Return Times Theorem (\textbf{Theorem \ref{RTT}}) to a more general situation with an infinite measure space.  In Lacey \cite{LaceyTRT} the following theorem is used to show that there exists a sigma-finite measure preserving system $(X, \mathcal{F}, \mu, T)$ and a set $A \subset X$ of positive finite measure so that for almost every $x \in X$ and for every aperiodic measure preserving system $(Y, \mathcal{G}, \nu, S)$ with $\nu(S) = 1$, there is a $g \in L^2(\nu)$ so that the averages
$$\tau_n^{-1}(x)\sum_{m=1}^n \chi_A(T^mx)g(S^my)$$
diverge for $\nu$-a.e. $y$ where
$$\tau_n(x) = \sum_{m=1}^n \chi_A(T^nx).$$

\begin{theorem}
Let $X_m$ be non-negative i.i.d. integer-valued random variables such that $P(X_1 > \lambda) \sim \lambda^{-\alpha}$ as $\lambda \rightarrow +\infty$.  Here $0 < \alpha < 1$, so that $EX_1 = +\infty$.  Then with probability 1, for every aperiodic finite measure-preserving system $(Y, \mathcal{G}, \nu, S)$ there is a square-integrable function $g$ on $Y$ for which
$$A_n(g)(y) = \frac{1}{N}\sum_{n=1}^Ng(S^{\tau_n}y)$$
diverges for $\nu$-a.e. $y$, where the power of $S$ above is $\tau_n = \sum_{m=1}^nX_m$.
\end{theorem}

This construction builds off of a creation of a simple random walk whose returns of the walk to the origin is almost surely a bad sequence along which to try the pointwise ergodic theorem.

\subsection{Recent Extensions}

In \cite{HK-US}, B. Host and B.  Kra have extended the Wiener-Wintner Theorem (\textbf{Theorem \ref{WW}}) by showing the following.

\begin{theorem}\label{HKNew}
Let $(X,\mathcal{F}, \mu, T)$ be an ergodic system and $f\in L^1(\mu)$  Then there is a set of full measure $X_f$ such that for every $x\in X_f$ the averages
$$\frac{1}{N}\sum_{n=1}^N a_nf(T^nx)$$
converge for every nilsequence $(a_n)$.
\end{theorem}

B. Host and  B. Kra. used \textbf{Theorem \ref{HKNew}}  to also prove that the sequence $(f(T^nx))$ is a.e. a good weight for the convergence in $L^2$-norm of the averages
$$\frac{1}{N}\sum_{n=1}^N (f(T^nx))\prod_{k=1}^K g_k\circ T^{kn}$$
extending their result in cite{HK-NEA} where $g_k\in L^{\infty}$.  The idea of mixing weights from a.e. multiple recurrence and the multiple return time theorem was introduced in \cite{AssaniMRT} for weakly mixing systems.

Note that a sequence $(a_n) \in \ell^{\infty}$ is a basic $l$-step nilsequence if there exists a basic $l$-step nilsystem $(G/\Gamma, S)$, a point $y \in G/\Gamma$, and a function $F \in C(G/\Gamma)$ such that $a_n = F(S^ny)$ for all $n \in \mathbb{N}$.  In \cite{E12}, a new proof of the result of Host and Kra is given by T. Eisner and P. Zorin-Kranich and extended to F\o lner sequences.

\subsection{Wiener-Wintner Dynamical Functions}

\begin{defn}
Given an ergodic dynamical system $(X, \mathcal{F}, \mu, T)$ a function $f\in L^p(\mu)$ is said to be a \textbf{Wiener-Wintner function of power type $\alpha$} if we can find a constant $C$ such that
$$\left\|\sup_t\left|\frac{1}{N}\sum_{n=1}^N f(T^nx) e^{2\pi int}\right|\right\|_1 \leq  \frac{C}{N^{\alpha}}$$
for each $N$.
\end{defn}

Ergodic systems having a dense set of Wiener-Wintner functions in $\mathcal{K}^{\perp}$ are called Wiener-Wintner dynamical systems (see \cite{AssaniWWD}).  Such systems allow one to give simple proofs of the return times theorem, the a.e double recurrence \cite{BourgainDRA}, and the convergence of the one sided ergodic Hilbert transforms
$$\sum_{n=1}^{\infty} \frac{f(T^nx) g(S^ny)}{n}$$
and
$$\sum_{n=1}^{\infty} \frac{f(T^{-n}x) g(S^{-n}y)}{n}.$$

For example, let us look at a proof for why the one sided Hilbert transform converges in the setting of Wiener-Wintner functions.

\begin{proof}
Take $f$ a Wiener-Wintner function of power type $\alpha>0$ and $\beta>0$ such that $\beta\alpha>1.$
We have
$$\int \sum_{N=1}^{\infty}\left\|\sup_t\left| \frac{1}{N^{\beta}}\sum_{n=1}^{N^{\beta}} f(T^nx) e^{2\pi int}\right|\right\|_1 \leq C \sum_{N=1}^{\infty} \frac{1}{N^{\beta\alpha}}<\infty.$$
We can fix $x\in X_f$, a set of full measure, such that
$$\sum_{N=1}^{\infty}\sup_t\left| \frac{1}{N^{\beta}}\sum_{n=1}^{N^{\beta}} f(T^nx) e^{2\pi int}\right|<\infty.$$
This set is independent of the dynamical system $(Y, \mathcal{G}, \nu).$ By the spectral theorem we have for each positive integer $N$
$$\int \left|\frac{1}{N}\sum_{n=1}^N f(T^nx) g(S^ny)\right|d\nu \leq C\sup_{\theta}\left|\frac{1}{N}\sum_{n=1}^N f(T^nx) e^{2\pi in\theta}\right|.$$
Therefore for  $x\in X_f$ we derive the following
$$\sum_{N=1}^{\infty} \left|\frac{1}{N^{\beta}}\sum_{n=1}^{N^{\beta}} f(T^nx) g(S^ny)\right|<\infty.$$
This implies that for $\nu$-a.e $y$
$$\frac{1}{N^{\beta}}\sum_{n=1}^{N^{\beta}} f(T^nx) g(S^ny)$$
converges to zero.

For the general sequence we consider for each integer $M$ the unique $N$ integer such that $N^\beta \leq M < (N+1)^{\beta}.$  Then we can write
$$\frac{1}{M}\sum_{n=1}^M f(T^nx)g(S^ny) = \frac{N^{\beta}}{M}\frac{1}{N^{\beta}}\sum_{n=1}^{N^{\beta}}f(T^nx)g(S^ny) + \frac{1}{M}\sum_{n=N^{\beta}}^Mf(T^nx)g(S^ny).$$
The last term goes to zero as it is dominated in absolute value by
$$\|g\|_{\infty}\frac{1}{M}\sum_{n=N^\beta}^{(N+1)^{\beta}}|f(T^nx)|$$
(in the case where $g\in L^{\infty}$) which is equal to
$$\frac{1}{M}\sum_{n=1}^{(N+1)^\beta}|f(T^nx)| - \frac{1}{M} \sum_{n=1}^{N^\beta}|f(T^nx)|.$$ And this last quantity goes to zero by the pointwise ergodic theorem for the function $|f|$ and the fact that the limit of $\frac{M}{N^{\beta}}$ is one.

The same argument works for the one sided series
$$\sum_{n=1}^{\infty}\frac{f(T^{-n}x)g(S^ny)}{n}$$
because for a $f$ which is a Wiener-Wintner function of power type $\alpha$ we also have for each $N$
$$\left\|\sup_t\left|\frac{1}{N}\sum_{n=1}^N f(T^{-n}x)e^{2\pi int}\right|\right\|_1 \leq \frac{C}{N^{\alpha}}.$$
\end{proof}

With some extra work using a truncation method one can prove the above convergence (a.e. $\nu$) for functions $g\in L^r(\nu)$ for $1<r\leq \infty.$

As shown in \cite{AssaniSCO} not all ergodic systems have such a dense set of functions in $\mathcal{K}^{\perp}$ even with a rate as slow as a logarithm. But interesting systems like $K$-automorphisms do.

\section{Conclusion}

For interested readers, there are several possible directions of study concerning return times to consider in addition to those mentioned above.

For the Cesaro averages one could extend this study to any good averaging process. For instance consider an increasing sequence of natural numbers $p(n)$ which is good for the pointwise convergence in $L^p$ (such as the sequence of squares in $L^p$ (for $1 < p \leq \infty$ \cite{BourgainDRA}).

\begin{ques}
Do the averages $\frac{1}{N}\sum_{n=1}^N f(T^{p(n)}x) g(S^ny)$ converge a.e. $\nu$?
\end{ques}

\begin{ques}\label{Q2}
Do the averages $ \frac{1}{N}\sum_{n=1}^N f(T^{p(n)}x)g(S^{p(n)}y)$ converge a.e. $\nu$?
\end{ques}

\begin{ques} For the previous two questions, what is the range of functions for which the result is true?
\end{ques}

\begin{ques}
Can one have a multiple term return times result for the averages above?
\end{ques}

\begin{ques}\label{Q5}
What would be the characteristic factors for the averages in referenced in the above questions?
\end{ques}

One could also look at the same questions translated to the corresponding Hilbert Transform.  Note that the notion of characteristic factors is not the same for the Hilbert transform and the Cesaro averages. One has to think instead in terms of the a.e. continuity of the limit of the series. For instance for the averages in \textbf{Question \ref{Q2}} the corresponding averages to study with respect to \textbf{Question \ref{Q5}} would be the a.e continuity in $t$ of the series
$$\sum_{n=-\infty}^{\infty} \frac{f(T^{p(n)x}e^{2\pi ip(n)t}}{n}$$
once the function $f$ is orthogonal to the appropriate factor.

Some of these questions have been raised by the first author during problem sessions at the Ergodic Theory Workshops he has organized yearly at The University of North Carolina at Chapel Hill since the summer of 2002.

\bibliographystyle{amsalpha}
\newcommand{\etalchar}[1]{$^{#1}$}
\providecommand{\bysame}{\leavevmode\hbox to3em{\hrulefill}\thinspace}
\providecommand{\MR}{\relax\ifhmode\unskip\space\fi MR }
\providecommand{\MRhref}[2]{%
  \href{http://www.ams.org/mathscinet-getitem?mr=#1}{#2}
}
\providecommand{\href}[2]{#2}

\end{document}